\documentclass[a4paper,11pt,reqno]{amsart} 
\usepackage{amssymb,amsthm,amsmath}
\usepackage{mathtools}
\usepackage{amsfonts}
\usepackage{amscd}
\usepackage{amssymb}

\usepackage{enumerate}
\usepackage{bbm}
\usepackage{graphicx}
\allowdisplaybreaks
\usepackage{mathabx}
\usepackage{color}
\usepackage{amsbsy}
\usepackage{graphicx}
\usepackage{amsthm}
\usepackage{amsmath}
\usepackage{amsxtra}
\usepackage{mathrsfs}
\usepackage{bbm}
\usepackage{dsfont}
\usepackage{enumitem}
\usepackage{comment}

\usepackage{enumerate}
\usepackage{xcolor}
\usepackage{float} 

\usepackage{url}
\usepackage{soul}
\usepackage[hidelinks]{hyperref} 

\usepackage{ifthen}

\hypersetup{
    colorlinks=true,
    linkcolor=red,      
    citecolor=red,       
    urlcolor=cyan        
}

\newtheorem{theorem}{Theorem}[section]

\newtheorem{thm}{Theorem}[subsection]
\newtheorem{defn}[thm]{Definition}
\newtheorem{remark}[theorem]{Remark}
\newtheorem{rmk}[thm]{Remark}

\newtheorem{proposition}[theorem]{Proposition}

\numberwithin{equation}{section}
\definecolor{red}{rgb}{1.0, 0.0, 0.0}
\newcommand{\Bea}{\begin{eqnarray*}}
	\newcommand{\Eea}{\end{eqnarray*}}
\newcommand{\Be} {\begin{equation*}}
	\newcommand{\Ee} {\end{equation*}}
\newcommand{\be} {\begin{equation}}
	\newcommand{\ee} {\end{equation}}
\newcommand{\bea} {\begin{eqnarray}}
	\newcommand{\eea} {\end{eqnarray}}

\newcommand{\al}{\alpha}

\newcommand{\la}{\lambda}

\usepackage{color}

\title[ Weighted Norm Inequalities on the Heisenberg Group ]{ Weighted Norm Inequalities for the Strichartz Fourier transform on the Heisenberg Group}
\author[A. Dasgupta]{Aparajita Dasgupta}
\address{
	Aparajita Dasgupta:
	\endgraf
	Department of Mathematics
	\endgraf
	Indian Institute of Technology Delhi, Hauz Khas
	\endgraf
	New Delhi-110016 
	\endgraf
	India
	\endgraf
	{\it E-mail address:} {\rm adasgupta@maths.iitd.ac.in}
}
\author[P. Gulia]{Prerna Gulia}
\address{
	Prerna Gulia:
	\endgraf
	Department of Mathematics
	\endgraf
	Indian Institute of Technology Delhi, Hauz Khas
	\endgraf
	New Delhi-110016 
	\endgraf
	India
	\endgraf
	{\it E-mail address:} {\rm prernagulia64@gmail.com}
}
\author[S. Pusti]{Sanjoy Pusti}
\address{
	Sanjoy Pusti:
	\endgraf
	Department of Mathematics
	\endgraf
	Indian Institute of Technology Bombay
	\endgraf
	Mumbai-400076 
	\endgraf
	India
	\endgraf
	{\it E-mail address:} {\rm sanjoy@math.iitb.ac.in}
}
\author[S. Thangavelu]{Sundaram Thangavelu}
\address{
		Sundaram Thangavelu:
	\endgraf
	National Science Chair, Department of Mathematics
	\endgraf
	Indian Institute of Science, Bangalore
	\endgraf
	Bangalore-560012 
	\endgraf
	India
	\endgraf
    and
    \endgraf
	Honorary Professor
	\endgraf
	University of Queensland, Brisbane
	\endgraf
    Australia
    \endgraf
	{\it E-mail address:} {\rm veluma@iisc.ac.in}
}


\vspace{1cm}

\allowdisplaybreaks
\begin{document}
\begin{abstract}
    In this article, we establish an analogue of Pitt’s inequality for the Strichartz Fourier transform on the Heisenberg group $\mathbb{H}^n$. By exploiting the scalar-valued formulation of the transform and the framework of decreasing rearrangements, we derive weighted $L^p$–$L^q$ estimates of Pitt type. In particular, we obtain sufficient conditions for the validity of such inequalities via weighted Hardy inequalities and Calderón’s interpolation method, and we also prove necessary conditions in the case of radial weights, using structural properties of Laguerre functions and zeros of Bessel function. As an application, we deduce an uncertainty principle of Heisenberg–Pauli–Weyl type in this setting and establish a Paley inequality for the Strichartz Fourier transform. We also derive Pitt's inequality using Hardy's inequality for the case \(p=q=2\). These results extend the classical Euclidean theory of Pitt’s inequality to the non-commutative, nilpotent setting of $\mathbb{H}^n$ for the sub-Laplacian and conformal Laplacian. Here we highlight the role of Laguerre functions in harmonic analysis on the Heisenberg group.
\end{abstract}
\maketitle
\section{Introduction}

The study of inequalities associated with the Fourier transform is a central theme in harmonic analysis. Such inequalities not only provide quantitative estimates relating a function and its Fourier transform, but also serve as essential tools in areas ranging from number theory to partial differential equations. A classical example is the Hausdorff–Young inequality, which shows that the Fourier transform is bounded from $L^p(\mathbb{R}^n)$ into $L^{p'}(\mathbb{R}^n)$ for $1 \leq p \leq 2$, where $p'$ denotes the Hölder conjugate of $p$. This result reveals the deep interplay between integrability properties in the time and frequency domains and serves as a prototype for a large family of Fourier inequalities.

A different yet closely related direction was initiated in 1937 by Pitt \cite{Pitts-1937}, who proved a weighted $L^p$–$L^q$ inequality for Fourier coefficients of integrable functions on the unit circle. Pitt’s inequality illustrates how weights can control the distribution of mass between a function and its Fourier transform, balancing decay at infinity against local integrability. Over time, this result has been extended and refined in many directions. In particular, Heinig \cite{heinig_1984} considered the problem in a general Euclidean setting, formulating necessary and sufficient conditions on pairs of non-negative weights $(u,v)$ for which the weighted inequality
\begin{equation}\label{weighted for sublinear}
    \left( \;\int\limits_{\mathbb{R}^n}|u(\xi)Tf(\xi)|^qd\xi\right)^{1/q}\leq C \left(\; \int\limits_{\mathbb{R}^n}|v(x)f(x)|^pdx\right)^{1/p},
\end{equation}
holds for a sublinear operator \(T\) defined on suitable function spaces where \(C\) is a constant which is independent of the choice of \(f\).

By choosing \(T\) to be the Fourier transform on \(\mathbb{R}^n\), inequality \eqref{weighted for sublinear} takes the form  
\begin{equation}\label{pitts inequality Rn}
    \left( \int\limits_{\mathbb{R}^n} |u(\xi)\,\widehat{f}(\xi)|^q \, d\xi \right)^{1/q} 
    \leq C \left( \int\limits_{\mathbb{R}^n} |v(x) f(x)|^p \, dx \right)^{1/p}.
\end{equation}  

A natural way to investigate the validity of \eqref{pitts inequality Rn} is via decreasing rearrangements. Denote by \(u^*\) the decreasing rearrangement of the non-negative weight \(u\) on \(\mathbb{R}^n\), and set \(U = u^*\). Similarly, let \(1/V = (1/v)^*\). In this framework, Heinig \cite{heinig_1984} established the following fundamental result.  

\begin{theorem} 
Let \(1 \leq p \leq q \leq \infty\). If \(p < \infty\) and \((p,q) \neq (2,2)\), and if  
\begin{equation}\label{Rn pitts cond1}
    \sup_{s > 0} \left( \int\limits_{0}^{1/s} U(t)^q \, dt \right)^{1/q} 
                   \left( \int\limits_{0}^{s} V(t)^{-p'} \, dt \right)^{1/p'} < \infty,  
\end{equation}  
then inequality \eqref{pitts inequality Rn} holds.  Moreover, in the special case \(p = q = 2\), if condition \eqref{Rn pitts cond1} is satisfied together with  
\begin{equation}\label{Rn pitts cond2}
    \sup_{s > 0} \left( \int\limits_{1/s}^{\infty} U(t)^2 \, t^{-1} \, dt \right)^{1/2} 
                   \left( \int\limits_{s}^{\infty} V(t)^{-2} \, t^{-1} \, dt \right)^{1/2} < \infty,  
\end{equation}  
then \eqref{pitts inequality Rn} remains valid.  
\end{theorem}  

In the same paper, Heinig also provided a necessary condition, showing that Pitt’s inequality can only hold under additional structural assumptions on the weights.  

\begin{theorem}
Suppose \(1 < p, q < \infty\) and that \eqref{pitts inequality Rn} is valid for two radial weights \(u\) and \(v\) on \(\mathbb{R}^n\). Then  
\[
    \sup_{s > 0} \left( \int\limits_{|\xi| < \theta_0/s} u(\xi)^q \, d\xi \right)^{1/q} 
                  \left( \int\limits_{|x| < s} v(x)^{-p'} \, dx \right)^{1/p'} < \infty,
\]  
where \(\theta_0\) denotes a positive constant strictly smaller than the first zero of the Bessel function \(j_{n/2}\).  
\end{theorem}  

Here one can notice that a particular important consequence of these general criteria arises when one specializes to power weights. Taking \(u(\xi) = |\xi|^{-\sigma}\) and \(v(x) = |x|^{\alpha}\), one recovers the classical Pitt inequality on $\mathbb{R}^n$ (see \cite{benedetto-heinig-weighted-inequalities, Pitts-1937}): for \(1 < p \leq q < \infty\),  
\begin{equation}\label{classical Pitts}
    \left( \int\limits_{\mathbb{R}^n} |\xi|^{-\sigma q} |\widehat{f}(\xi)|^q \, d\xi \right)^{1/q} 
    \leq C \left( \int\limits_{\mathbb{R}^n} |x|^{\alpha p} |f(x)|^p \, dx \right)^{1/p},    
\end{equation}  
which holds if and only if  
\[
   0 \leq \sigma < \frac{n}{q}, 
   \qquad 0 \leq \alpha < \frac{n}{p'}, 
   \qquad \text{and} \qquad 
   \alpha - \sigma = n \left( 1 - \frac{1}{p} - \frac{1}{q} \right).
\]  

This formulation not only unifies Pitt’s original inequality with its weighted generalizations but also situates it alongside other fundamental Fourier inequalities, such as Hausdorff–Young and Paley’s inequalities, as part of a broader framework in harmonic analysis.  
Paley's inequality is another important inequality associated with the Fourier transform, which concerns the question of the existence of a measure \(\nu\) such that the Fourier transform is \(L^p-L^p_\nu\) bounded. In the Euclidean context, H\"ormander \cite{Hormander-1960-paley} pioneered the study of such inequality. Over the years, weighted norm inequalities have inspired formulations across different mathematical structures, resulting in an extensive literature in this field. One can refer to \cite{Beckner-Pitts_and_uncertainty_principle_1995,   beckner-Pitts_with_sharp_convolution_estimates, benedetto-heinig-weighted-inequalities, Pitts_and_Boas_for_Fourier_and_Hankel,  Pitts_for_quaternion_fourier_transform,Liflyand-Tikhonov} and the references therein for a detailed overview of the scope of such weighted inequalities. 

In 2024, Kumar et al. \cite{weighted-fourier-symmetric-spaces-Pusti} studied weighted \(L^p-L^q\) inequalities for the Fourier transform in the setup of rank one Riemannian symmetric spaces of non-compact type. These results highlight the role of Fourier-analytic techniques adapted to non-Euclidean geometric structures. Motivated by the classical Helgason Fourier transform on symmetric spaces, in 2023, the fourth author \cite{Strichartz_Fourier_thangavelu} proposed a definition for a scalar-valued Fourier transform on the Heisenberg group. This transform is known as the Strichartz Fourier transform. It is defined using the joint eigenfunctions of the sublaplacian \(\mathcal{L}\) and the vector field \(T=\frac{\partial}{\partial t}\) of the Heisenberg group.

The Heisenberg group \(\mathbb{H}^n\), is the simplest example of a non-abelian nilpotent Lie group whose underlying manifold is $\mathbb{C}^n \times \mathbb{R}$, equipped with the following group operation
 $$
(z,t)(w,s)=\left( z+w, t+s+\frac{1}{2}\operatorname{Im}(z \cdot \bar{w})\right),
 $$
 where \((z,t),(w,s) \in \mathbb{C}^n \times \mathbb{R}\).
 For a function \(f\) on \(\mathbb{H}^n\), its Fourier transform (in usual sense) is an operator on \(L^2(\mathbb{R}^n)\) defined as 
 $$
 \widehat{f}(\lambda)= \int\limits_{\mathbb{C}^n \times \mathbb{R}}f(z,t)\pi_{\lambda}(z,t)\;dz\;dt,
 $$
 where \(\pi_{\lambda}\) denotes the Schr\"odinger representation of \(\mathbb{H}^n\) on \(L^2(\mathbb{R}^n)\). More explicitly, for \(\varphi \in L^2(\mathbb{R}^n)\),
 $$
 \pi_\lambda(z,t)\varphi(\xi)=e^{i\lambda t}e^{i \lambda (x \cdot \xi+\frac{1}{2}x \cdot y) }\varphi(\xi+y),
 $$
 where \(z=x+iy\). From the definition, it is easy to see that 
 $$\widehat{f}(\lambda)=\int\limits_{\mathbb{C}^n}f^\lambda(z)\pi_\lambda(z,0)\,dz,$$
 where \(f^{\lambda}\) stands for the inverse Fourier transform of \(f\) in the center variable, that is, 
 $$
 f^{\lambda}(z)=\int\limits_{-\infty}^{\infty}f(z,t)e^{i\lambda t}\,dt.
 $$
 In the Heisenberg context, the group Fourier transform is operator-valued and, although it satisfies Plancherel and inversion formulas, its complexity makes it less suitable for addressing problems such as the characterization of the Schwartz class or the formulation of weighted inequalities. This difficulty is overcome by the Strichartz Fourier transform, introduced in \cite{Strichartz_Fourier_thangavelu}, which provides a scalar-valued analogue sharing several structural features with the Helgason Fourier transform on symmetric spaces. This transform enables explicit descriptions of images of test functions via the Hecke–Bochner formula and allows one to restate classical theorems of Hardy and Ingham in this new framework. 

Building on this development, our goal in this article is to extend the study of Strichartz Fourier transform. The novelty of this work lies in establishing  
weighted $L^p$--$L^q$ inequalities of Pitt type for the Strichartz Fourier transform. 
The scalar-valued formulation proves crucial here, as it allows the use of decreasing rearrangements and interpolation methods, that are unavailable in the operator-valued setting. We use these methods to obtain sufficient conditions on weights for the Pitt's inequality on \(\mathbb{H}^n\), which, to the best of our knowledge, have not appeared previously in the literature. We further derive necessary conditions on pairs of weight for these inequalities to hold, 
exploiting the structural role of Laguerre functions in the Fourier analysis on the Heisenberg group.
We also investigate the connection between Pitt’s inequality and Hardy’s inequality. After recalling their equivalence in the Euclidean setting, we extend this relationship to the Heisenberg group for the case \(p=q=2\). In particular, we show how Pitt-type inequalities (and corresponding Hardy inequalities) are obtained for the conformal sublaplacian for different weights. We also derive analogues of Pitt's inequality for the non-homogeneous weight and Pitt's inequality for the fractional sublaplacian. Reformulating these results in terms of the Strichartz Fourier transform provides a unified framework for understanding weighted uncertainty-type inequalities on \(\mathbb{H}^n.\)
 Although the optimality of the constants obtained is not addressed in this work.


 We conclude the introduction by providing a brief outline of the article. In section \ref{sec 2}, we present essential background about the Strichartz Fourier transform, recalling the Plancherel formula, inversion formula and the Hausdorff-Young inequality. To establish the Pitt's inequality, we also need the concept of decreasing rearrangement of a function on a metric space and the weighted Hardy inequalities which have been described briefly in the section \ref{sec 2}. In sections \ref{sec 3} and \ref{sec 4}, we prove Pitt's inequality for the Strichartz Fourier transform and derive necessary conditions on the weights. Moreover we also deduce an uncertainty principle of Heisenberg–Pauli–Weyl type as a theorem. Subsequently, in section \ref{sec 5}, we give a non-trivial example of power weights for the theorem to hold and we deduce sufficient and necessary conditions on power weights. In section \ref{sec 6}, we prove a version of Paley’s inequality within the framework of the Strichartz Fourier transform. Finally, in section \ref{sec 7} we study Pitt's inequality via Hardy's inequality for usual operator valued Fourier transform on Heisenberg group.

\section{Preliminaries}\label{sec 2}
In this section, we recall basic preliminaries required to understand the Strichartz Fourier transform. Furthermore, we will also recall its associated Plancherel theorem, inversion formula and the Hausdorff-Young inequality.

For \(\delta >-1\) and \(k \in \mathbb{N} \cup \{0\}\), let \(L^{\delta}_k\) denote the Laguerre polynomial of type \(\delta\) and degree \(k\), defined as
$$
L^{\delta}_k(t)=\frac{t^{-\delta}e^t}{k!}\left ( \frac{d}{dt} \right )^k \left( e^{-t}t^{k+\delta}\right),\; t\geq 0.
$$
Given \(\lambda \in \mathbb{R}^*\), let \(\varphi^{n-1}_{k,\lambda}\) denote the Laguerre function \(\mathbb{C}^n\) defined as 
$$
\varphi^{n-1}_{k,\lambda}(z)=L^{n-1}_k\left( \frac{1}{2}|\lambda||z|^2\right) e^{-\frac{1}{4}|\lambda||z|^2}.
$$
For more details about the Laguerre functions, one can refer to  \cite{thangaveluheisenberg, wong_weyl-transforms}.

The Laguerre functions play an important role in determining the joint eigenfunctions of the sublaplacian \(\mathcal{L}\) on \(\mathbb{H}^n\) and the left invariant vector field \(\frac{\partial}{\partial t}\). To make the paper self-contained, we briefly recall the joint spectral theory of \(\mathcal{L}\) and \(T\). The basis of the Heisenberg algebra \(\mathfrak{h}_n\) consists of the left invariant vector fields
\begin{equation*}
    X_i=\frac{\partial}{\partial x_i}+\frac{1}{2}y_i\frac{\partial}{\partial t},\; Y_i=\frac{\partial}{\partial y_i}-\frac{1}{2}x_i\frac{\partial}{\partial t},\; T= \frac{\partial}{\partial t},\; i=1,2,\dots , n.
\end{equation*}
It is known that the sublaplacian \(\mathcal{L}=-\sum\limits_{i=1}^n\left( X_i^2+Y_i^2 \right)\) commutes with the vector field \(T\). This enables the study of the joint spectral theory of these two operators. 

From \cite{thangaveluheisenberg}, we have that the functions \(e^{n-1}_{k,\lambda}(z,t)=e^{i \lambda t}\varphi^{n-1}_{k,\lambda}(z)\) satisfy the following eigenvalue equations
$$
\mathcal{L}(e^{n-1}_{k,\lambda})=(2k+n)\lvert \lambda \rvert e^{n-1}_{k,\lambda} , \;\;\;-i\frac{\partial}{\partial t}(e^{n-1}_{k,\lambda})=\la e^{n-1}_{k,\lambda} .
$$
The operators \(L_{\lambda}\) defined by the relation \(\mathcal{L}(e^{i \lambda t}f(z)=e^{i\lambda t}L_{\lambda}f(z)\), are called special Hermite operators. The Laguerre functions \(\varphi^{n-1}_{k,\lambda}\) are eigenfunctions of \(L_\lambda\) with eigenvalue \((2k+n)|\lambda|\). Moreover, the eigenspace of \(L_\lambda\) corresponding to eigenvalue \((2k+n)|\lambda|\) is an infinite-dimensional space with an orthonormal basis consisting of functions 
\begin{equation*}
    \Phi^\lambda_{\alpha,\beta}=(2\pi)^{-n/2}\langle \pi_\lambda (z,0)\Phi^\lambda_\alpha, \Phi^\lambda_\beta \rangle,
\end{equation*}
where \(\alpha,\beta \in \mathbb{N}^n,\;|\alpha|=k\) and \(\Phi^\lambda_\alpha\) are scaled Herimte functions which form eigenfunctions of the scaled Hermite operator \(H(\lambda)=-\Delta+\lambda^2|x|^2\). For a more comprehensive study of eigenfunctions, we refer the reader to \cite{thangavelu-uncertainty-book}. 

Now, let \(\Omega\) denote the Heisenberg fan, which is defined as the union of rays 
$$
R_k=\{(\lambda,(2k+n)|\lambda|): \lambda \in \mathbb{R}^*\},\; k=0,1,\cdots ,
$$
and the limiting ray 
$$
R_{\infty}=\{(0,\tau): \tau \geq 0\}.
$$
For a function \(f \in L^1(\mathbb{H}^n) \cap L^2(\mathbb{H}^n)\), its Strichartz Fourier transform is a function on \(\Omega \times \mathbb{C}^n\), defined as
\begin{equation}\label{sft def1}
    \widehat{f}(a,w)= \int\limits_{\mathbb{H}^n} f(z,t) e_a\left ((z,t)^{-1}(w,0) \right) dz\;dt,
\end{equation}
where \(a=(\lambda, (2k+n)|\lambda|) \in R_k\) and \(e_a(z,t)=e^{n-1}_{k,\lambda}(z,t)\).
Moreover, for element \((0,\tau)\in R_{\infty}\), 
\begin{equation}
    \widehat{f}(0,\tau,w)= (n-1)! 2^{n-1}\int\limits_{\mathbb{H}^n}f(z,t)\frac{J_{n-1}(\sqrt{\tau}|w-z|)}{{(\sqrt{\tau}|w-z|)}^{n-1}}dzdt,
\end{equation}
where \(J_{n-1}\) refers to the Bessel function of order \((n-1)\).

We recall the inversion formula and Plancherel theorem for the Strichartz Fourier transform. To state the theorem, one needs a measure on the set \(\Omega\). Consider the measure \(\nu\) on \(\Omega\) defined as 
\begin{equation}
    \int\limits_{\Omega} \phi(a)d\nu(a)=(2\pi)^{-2n-1}\int\limits_{-\infty}^{\infty}\left (\sum_{k=0}^{\infty}\phi(\lambda,(2k+n)|\lambda|) \right )|\lambda|^{2n}d\lambda.
\end{equation}
\begin{theorem}
    For \(f \in \mathcal{S}(\mathbb{H}^n)\), the inversion formula for the Strichartz Fourier transform \cite{Strichartz_Fourier_thangavelu} states that 
    \begin{equation}
        f(z,t)=\int\limits_{\Omega}\int\limits_{\mathbb{C}^n}\widehat{f}(a,w)e_a((-w,0)(z,t)dw\;d\nu(a).
    \end{equation}
Furthermore, for a function \(f\) in \(L^1 \cap L^2(\mathbb{H}^n)\), we have the following Plancherel formula
\begin{equation}
    \int\limits_{\mathbb{H}^n}|f(z,t)|^2dz\;dt=\int\limits_{\Omega}\int\limits_{\mathbb{C}^n}|\widehat{f}(a,w)|^2dw\;d\nu(a).
\end{equation}
\end{theorem}

The normalised Strichartz Fourier transform \(\tilde{f}\) is defined as 
$$
\tilde{f}(a,w)=\frac{k!(n-1)!}{(k+n-1)!}\widehat{f}(a,w),
$$
and
$$
\tilde{f}(0,\tau,w)=\widehat{f}(0,\tau,w).
$$
Using the following estimates of the Laguerre function and the Bessel function (see \cite{handbook,szego})
$$
\frac{k!(n-1)!}{(k+n-1)!}|\varphi^{n-1}_{k,\lambda}(z)| \leq 1 \text{ and } |J_{n-1}(t)| \leq c_n t^{n-1},
$$
one obtains that for any \(f \in L^1(\mathbb{H}^n)\), the normalized Strichartz Fourier transform satisfies 
\begin{equation}\label{L1 to L inf bdd}
  \sup_{(a,w)\in \Omega \times \mathbb{C}^n} |\tilde{f}(a,w)|\leq c_{n,1}\lVert f\rVert_1. 
\end{equation}

Furthermore, restating the Plancherel formula in terms of the normalized Strichartz Fourier transform, we have 
\begin{equation}\label{plancherel}
    \int\limits_{\mathbb{H}^n}|f(z,t)|^2dz\;dt=\int\limits_{\Omega}\int\limits_{\mathbb{C}^n}|\tilde{f}(a,w)|^2dw\;d\nu_2(a), 
\end{equation}
where the measure \(d\nu_2(a)\) is a measure on \(\Omega\) given by 
$$
\int\limits_{\Omega} \phi(a)d\nu_2(a)=(2\pi)^{-2n-1}\int\limits_{-\infty}^{\infty}\left (\sum_{k=0}^{\infty}\left ( \frac{(k+n-1)!}{k!(n-1)!}\right )^2\phi(\lambda,(2k+n)|\lambda|) \right )|\lambda|^{2n}d\lambda.
$$
It is important to note that the Strichartz Fourier transform is not a unitary operator from \(L^2(\mathbb{H}^n)\) onto \(L^2(\Omega \times \mathbb{C}^n, d\nu\;dw)\), instead if we consider the space \(L_0^2(\Omega \times \mathbb{C}^n, d\nu\;dw)\) as the subspace of \(L^2(\Omega \times \mathbb{C}^n, d\nu\;dw)\) consisting of functions satisfying the property
\begin{equation}
    (2\pi)^{-n}|\la|^n\varphi^{n-1}_{k,\la} \ast_\la\widehat{f}(a,\cdot)(z)=\widehat{f}(a,z),\; a\in R_k
\end{equation}
then we have the following result.

\begin{theorem}
    The Strichartz Fourier transform is a unitary operator from \(L^2(\mathbb{H}^n)\) onto \(L_0^2(\Omega \times \mathbb{C}^n, d\nu\;dw)\).
\end{theorem}
The Pitt's inequality is formulated within the framework of the Hausdorff--Young inequality.  In the framework of the Strichartz Fourier transform, the Hausdorff–Young inequality, as established in   \cite{Strichartz_Fourier_thangavelu},  can be stated as follows.
\begin{theorem}
    Given \(1\leq p \leq 2\), there is a measure \(\nu_p\) on the set \(\Omega\) such that for every \(f \in L^p(\mathbb{H}^n)\), 
    \begin{equation}\label{H-Y inequality}
\left (\int\limits_{\Omega}\int\limits_{\mathbb{C}^n}|\tilde{f}(a,w)|^{p'}dw\;d\nu_p(a) \right )^{1/p'} \leq c_{n,p}\left(\int\limits_{\mathbb{H}^n}|f(z,t)|^pdz\;dt \right)^{1/p}.
    \end{equation}
\end{theorem}
This is established using the fact that the Strichartz Fourier transform is bounded from \(L^1(\mathbb{H}^n)\) into \(L^{\infty}(\Omega \times \mathbb{C}^n, d\nu(a)\;dw)\), and is also bounded from \(L^2(\mathbb{H}^n)\) into \(L^2(\Omega \times \mathbb{C}^n, d\nu_2(a)\;dw)\) and then use the interpolation theorem with change of measures. 
Using \cite[Remark 7.2.1]{Strichartz_Fourier_thangavelu}, we also have that 
\begin{equation}\label{H-Y without change of measures}
\left (\int\limits_{\Omega}\int\limits_{\mathbb{C}^n}|\tilde{f}(a,w)|^{p'}dw\;d\nu_2(a) \right )^{1/p'} \leq c'_{n,p}\left(\int\limits_{\mathbb{H}^n}|f(z,t)|^pdz\;dt \right)^{1/p}.
    \end{equation}
While the inequality in \eqref{H-Y inequality} provides a sharper estimate than \eqref{H-Y without change of measures}, for the purpose of deriving an analogue of Pitt’s inequality we shall adopt \eqref{H-Y without change of measures} as the basis of our analysis. To make the paper self-contained we recall here the decreasing rearrangement and Lorentz spaces.
\subsection{Decreasing rearrangement and Lorentz spaces}
Let \((X,\mu)\) be a measure space and \(f\) be a measurable function on \(X\). The \emph{distribution function} of \(f\), denoted by \(d_f\) is a function on \([0,\infty)\) defined as 
$$
d_{f}(M)=\mu \left \{x: |f(x)|>M\right \}.
$$
\noindent Given a measurable function \(f\), its \emph{decreasing rearrangement} \(f^*\) is a function defined on \([0,\infty)\) which is decreasing and satisfies \(d_{f}=d_{f^*}\). Explicitly, for \(t \geq 0\), 
$$
f^*(t)=\inf \left \{s>0:d_{f}(s)\leq t\right \}.
$$
For details and further properties of decreasing rearrangement, we refer the reader to \cite{grafakos_book}. Having introduced decreasing rearrangements of functions, we move to the definition of the Lorentz spaces.
\begin{defn}
    Let \(f\) be a measurable function on \((X,\mu)\) and \(0 <p,q \leq \infty\), define
    $$
    \lVert f \rVert_{L^{p,q}}=\begin{cases}
    \left (\displaystyle\int\limits_{0}^{\infty}\left( t^{1/p}f^*(t)\right)^q\frac{dt}{t} \right)^{1/q}\; &\textrm{if}\; q< \infty \\
    \sup\limits_{t>0}t^{1/p}f^*(t)\; &\textrm{if}\; q=\infty.
\end{cases} .
    $$
    The set of all the functions \(f\) for which \({\lVert f \rVert}_{L^{p,q}}< \infty\) is called the Lorentz space with indices \(p\) and \(q\) and is denoted by \(L^{p,q}\).
\end{defn}
We state a few properties (check details in \cite{bennett's_book,grafakos_book}) of the decreasing rearrangement and Lorentz space, which will be used in further sections. For two measurable functions \(f,g\) on \((X,\mu)\), 
\begin{equation}\label{norm in dec arr}
\displaystyle\int\limits_{X}|f|^pd\mu=\displaystyle\int\limits_{0}^{\infty}(f^*(t))^pdt\;\; \text{ for } 0<p<\infty,
\end{equation}
and
\begin{equation}\label{two func in dec arr}
\left (\int\limits_{0}^{\infty} \left(\frac{1}{(1/g)^*(t)}f^*(t)\right)^p dt\right )^{1/p} \leq \left (\int\limits_{X}|f(x)g(x)|^pd\mu \right)^{1/p}.
\end{equation}
The second inequality can be found in \cite{heinig_1984}.
The containment of Lorentz spaces is
as follows
$$
\text{for } 0<p \leq \infty \text{ and } 0<q <r \leq \infty, L^{p,q} \subseteq L^{p,r}.
$$
\subsection{Calder\'on's operator and Weighted Hardy's inequalities}
We now recall a few results related to Calder\'on's operator \cite{calderon66} and weighted Hardy's inequalities \cite{bradley-weighted-hardy}. Given \(1 \leq p_1 < p_2 \leq \infty\) and \(1 \leq q_1,q_2 \leq \infty\), where \(q_1 \neq q_2\), consider the segment \(\sigma\) with end points \(\left ( \frac{1}{p_1}, \frac{1}{q_1}\right)\) and \(\left (\frac{1}{p_2}, \frac{1}{q_2} \right)\). For each such segment, we define its slope as $$
m=\frac{1/q_1-1/q_2}{1/p_1-1/p_2}.
$$
Furthermore, associated to each such segment we have the corresponding \emph{Calder\'on operator} \(S_{\sigma}\) whose domain is the set of non-negative measurable functions \(f\) on \((0,\infty)\). For each \(t>0\), \(S_{\sigma}\) is defined as 
\begin{equation}\label{S_sigma equation}
S_{\sigma}f(t)=t^{-1/q_1}\int\limits_{0}^{t^m}s^{1/p_1-1}f(s)ds+t^{-1/q_2}\int\limits_{t^m}^{\infty}s^{1/p_2-1}f(s)ds.
\end{equation}
Let \((X,\mu)\) and \((Y,\nu)\) be two \(\sigma\)-finite measure spaces. A subilinear operator \(T\) acting on \(\mu\) measurable functions taking values in \(\nu\) measurable functions is said to be of \emph{joint weak type} \((p_1,q_1;p_2,q_2)\) if 
\begin{equation}
    (Tf)^*(t) \leq C S_{\sigma}f^*(t),\; \forall f \text{ and } t \in (0,\infty).
\end{equation}
Also, the operator \(T\) is said to be of \emph{separate weak types} \((p_1,q_1)\) and \((p_2,q_2)\) if T is bounded from \(L^{p_1,1} \text{ to } L^{q_1,\infty}\) and is bounded from
\(L^{p_2,1} \text{ to } L^{q_2,\infty}\). The following result \cite{bennett's_book,calderon66} gives the relation between \emph{joint weak type} and \emph{separate weak types}.

\begin{thm}\label{calderon's operator theorem}
    Let \(T\) be a sublinear operator on \(L^{p_1,1}(X,\mu)+L^{p_2,1}(X,\mu)\) taking values in the space of measurable functions on \((Y,\nu)\). Then T is of separate weak types \((p_1,q_1)\) and \((p_2,q_2)\) iff \(T\) is of joint weak type \((p_1,q_1;p_2,q_2)\).
\end{thm}

We also need the following weighted Hardy inequalities, whose proof can be found in \cite{bradley-weighted-hardy}.
\begin{thm}\label{weighted hardy}
    Let \(1 \leq p \leq q \leq \infty\) and \(u_1,\;v_1,\;f_1\) be non-negative functions. Then
    \begin{enumerate}
        \item[(a)] We have 
        $$
        \left (\displaystyle\int\limits_{0}^{\infty}  \left [ u_1(t) \displaystyle\int\limits_{0}^tf_1(s)ds\right ]^qdt \right )^{1/q} \leq C \left ( \displaystyle\int\limits_{0}^{\infty}\left[v_1(t)f_1(t) \right]^p dt \right)^{1/p},
        $$
        if and only if
        $$
        \sup\limits_{s>0}\left ( \displaystyle\int\limits_{s}^{\infty} u_1(t)^q dt\right)^{1/q}\left ( \displaystyle\int\limits_{0}^{s} v_1(t)^{-p'} dt\right)^{1/p'} < \infty.
        $$
        \item[(b)] Also, 
        $$
        \left (\displaystyle\int\limits_{0}^{\infty}  \left [ u_1(t) \displaystyle\int\limits_{t}^{\infty} f_1(s)ds\right ]^qdt \right )^{1/q} \leq C \left ( \displaystyle\int\limits_{0}^{\infty}\left[v_1(t)f_1(t) \right]^p dt \right)^{1/p},
        $$
        holds if and only if,
        $$
        \sup\limits_{s>0}\left ( \displaystyle\int\limits_{0}^{s} u_1(t)^q dt\right)^{1/q}\left ( \displaystyle\int\limits_{s}^{\infty} v_1(t)^{-p'} dt\right)^{1/p'} < \infty.
        $$
    \end{enumerate}
\end{thm}

\section{Pitt's Inequality for Strichartz Fourier Transform}\label{sec 3}

In this section, we prove an analogue of Pitt's inequality for the Heisenberg group.
From this point, \(u\) and \(v\) denote two non-negative measurable functions on \(\Omega \times \mathbb{C}^n\) and \(\mathbb{H}^n\) respectively. Let \(U=u^*\) be the decreasing rearrangement of \(u\) with respect to the measure \(d\nu_2(a) \;dw\) on \(\Omega \times \mathbb{C}^n\) and \(v^*\) be the decreasing rearrangement of \(v\) with respect to measure \(dz\;dt\) on \(\mathbb{H}^n\). Take \(\frac{1}{V}=\left (\frac{1}{v}\right )^*\). In what follows, \(C\) refers to a constant which is independent of the choice of the function and may have different values for different inequalities.

\begin{theorem}\label{theorem-Pitts for Hn}
   Let \(1 \leq p \leq q \leq \infty\). If \(p< \infty\) and 
    \begin{equation}\label{pitts cond1}
        \sup_{s>0}\left ( \int\limits_{0}^{1/s}U(t)^qdt \right )^{1/q} \left ( \int\limits_{0}^{s}V(t)^{-p'}dt \right )^{1/p'} <\infty, 
    \end{equation}
    then there exists a constant \(C\) such that for all \(f \in \mathcal{S}(\mathbb{H}^n)\),
    \begin{equation}\label{pitts inequality Hn}
        \left( \int\limits_{\Omega \times \mathbb{C}^n}|\tilde{f}(a,w)|^q u(a,w)^q dw\;d\nu_2(a)  \right )^{1/q}\leq C \left( \int\limits_{\mathbb{H}^n}|{f}(z,t)|^p v(z,t)^p dz\;dt  \right )^{1/p}.
    \end{equation}
\end{theorem}

\begin{proof}
    First consider the case when \(p <\infty\) and \((p,q)\neq (2,2)\).
   Consider the operator \(T(f)=\tilde{f}\) on \(L^1(\mathbb{H}^n)+L^{2,1} (\mathbb{H}^n)\). Note that \(L^{1,1}=L^1\) and \(L^{2,1} \subset L^{2,2}=L^2\). By \eqref{L1 to L inf bdd}, the operator \(T\) is strong type \((1,\infty)\) and so it is weak type \((1,\infty)\). Again from Plancherel theorem \eqref{plancherel}, \(T\) is weak type \((2,2)\).
  Applying Theorem \ref{calderon's operator theorem}, it follows that 
\begin{equation*}
   (Tf)^*(t) \leq C\, S_{\sigma} f^*(t),
\end{equation*}
for every \(t \in (0,\infty)\), where \(\sigma\) is the line segment with endpoints \((1,0)\) and \(\bigl(\tfrac{1}{2},\tfrac{1}{2}\bigr)\).

Now using the definition of operator \(S_{\sigma}\) as in \eqref{S_sigma equation}, we have 
\begin{equation}
   (\tilde{f})^*(t)= (Tf)^*(t)\leq C\left \{ \int\limits_{0}^{1/t}f^*(s)ds+ t^{-1/2}\int\limits_{1/t}^{\infty}s^{-1/2}f^*(s)ds \right \} 
\end{equation}

Multiplying both sides of the inequality by $U(t)$ and integrating with respect to $t$, we obtain
\begin{align}
   \left( \int\limits_{0}^{\infty} \left[ (\tilde{f})^*(t) U(t) \right]^q \, dt \right)^{1/q} 
   &\leq C \Bigg\{ 
   \left( \int\limits_{0}^{\infty} U(t)^q \left( \int\limits_{0}^{1/t} f^*(s)\, ds \right)^q dt \right)^{1/q} \nonumber\\
   &\quad+ \left( \int\limits_{0}^{\infty} U(t)^q t^{-q/2} 
   \left( \int\limits_{1/t}^{\infty} s^{-1/2} f^*(s)\, ds \right)^q dt \right)^{1/q} 
   \Bigg\}. \label{I_1+I_2}
\end{align}


Define  
\[
I_1 = \left( \int\limits_{0}^{\infty} U(t)^q \left( \int\limits_{0}^{1/t} f^*(s)\, ds \right)^q dt \right)^{1/q}.
\]
Applying the change of variables \(t \mapsto t^{-1}\), we get
\[
I_1 = \left( \int\limits_{0}^{\infty} U(t^{-1})^q\, t^{-2} 
        \left( \int\limits_{0}^{t} f^*(s)\, ds \right)^q dt \right)^{1/q}.
\]
Next, we invoke our hypothesis together with the weighted Hardy inequalities. 
In particular, by Theorem~\ref{weighted hardy}~(a), we obtain
\begin{equation}\label{estimate of I_1}
    I_1 \leq C \left( \int\limits_{0}^{\infty} \big( V(t) f^*(t) \big)^p \, dt \right)^{1/p},
\end{equation}
if and only if
\[
   \sup_{s>0} \left( \int\limits_{s}^{\infty} U(t^{-1})^q t^{-2} \, dt \right)^{1/q}
   \left( \int\limits_{0}^{s} V(t)^{-p'} \, dt \right)^{1/p'} < \infty.
\]

   
By a suitable change of variables, this condition coincides with our hypothesis \eqref{pitts cond1}.

   $$
   \sup\limits_{s>0}\left ( \displaystyle\int\limits_{0}^{1/s} U(t)^q dt\right)^{1/q}\left ( \displaystyle\int\limits_{0}^{s} V(t)^{-p'} dt\right)^{1/p'} < \infty.
   $$
   Now we work with the integral 
   $$
   I_2=\left[ \int\limits_{0}^{\infty}U(t)^q t^{-q/2} \left ( \int\limits_{1/t}^{\infty}s^{-1/2}f^*(s)ds\right)^qdt\right]^{1/q} .
   $$
   Again, by change of variables
   $$
   I_2=\left[ \int\limits_{0}^{\infty}U(t^{-1})^q t^{q/2}t^{-2} \left ( \int\limits_{t}^{\infty}s^{-1/2}f^*(s)ds\right)^qdt\right]^{1/q} .
   $$
Applying Theorem \ref{weighted hardy} \((b)\) with  
\[
u_1(t) = U(t^{-1})\, t^{1/2} t^{-2/q}, 
\quad v_1(t) = t^{1/2} V(t), 
\quad f_1(t) = t^{-1/2} f^*(t),
\] 
we deduce that  
\[
I_2 \leq C \left( \int\limits_{0}^{\infty} \bigl(V(t) f^*(t)\bigr)^p \, dt \right)^{1/p}
\]
if and only if  
\[
\sup_{s>0} 
\left( \int\limits_{0}^{s} U(t^{-1})^q\, t^{q/2} t^{-2} \, dt \right)^{1/q}
\left( \int\limits_{s}^{\infty} t^{-p'/2} V(t)^{-p'} \, dt \right)^{1/p'} < \infty.
\]
   $$
   $$
   
After performing the change of variables \(t \mapsto t^{-1}\), the condition becomes  
\begin{equation}\label{newcond_1}
  \sup_{s>0} 
  \left( \int\limits_{1/s}^{\infty} U(t)^q\, t^{-q/2} \, dt \right)^{1/q}
  \left( \int\limits_{s}^{\infty} t^{-p'/2} V(t)^{-p'} \, dt \right)^{1/p'} < \infty.
\end{equation}
However, in the hypothesis, we only have 
$$
\sup\limits_{s>0}\left ( \displaystyle\int\limits_{0}^{1/s} U(t)^q dt\right)^{1/q}\left ( \displaystyle\int\limits_{0}^{s} V(t)^{-p'} dt\right)^{1/p'} < \infty.
$$
Now, using \cite[Lemma 3.3]{weighted-fourier-symmetric-spaces-Pusti}, with \(N=1\) and \(q_0=2\), we obtain that \eqref{pitts cond1} implies \eqref{newcond_1}.  
Thus, we get 
\begin{equation}\label{estimate of I_2}
    I_2\leq C \left[ \int\limits_{0}^{\infty}\left( V(t)f^*(t)\right)^p\right]^{1/p}.
\end{equation}

Therefore, by combining \eqref{I_1+I_2}, \eqref{estimate of I_1}, and \eqref{estimate of I_2}, 
it follows that for \(p < \infty\) and \((p,q) \neq (2,2)\), there exists a constant \(C > 0\) such that  
\begin{equation}\label{combined estimate of integral}
    \left( \int\limits_{0}^{\infty} \bigl[(\tilde{f})^*(t) U(t)\bigr]^q \, dt \right)^{1/q} 
    \leq C \left( \int\limits_{0}^{\infty} \bigl[V(t) f^*(t)\bigr]^p \, dt \right)^{1/p}.
\end{equation}

\end{proof}

Next, applying the property \eqref{norm in dec arr} of decreasing rearrangements together with 
\eqref{combined estimate of integral}, we obtain  
\begin{align}
\left( \int\limits_{\Omega \times \mathbb{C}^n} |\tilde{f}(a,w)|^q \, u(a,w)^q \, dw \, d\nu_2(a) \right)^{1/q} 
&\leq \left( \int\limits_{0}^{\infty} \bigl[(\tilde{f})^*(t) U(t)\bigr]^q \, dt \right)^{1/q} \nonumber \\
&\leq C \left( \int\limits_{0}^{\infty} \bigl[V(t) f^*(t)\bigr]^p \, dt \right)^{1/p} \nonumber \\
&= C \left( \int\limits_{0}^{\infty} \left[ \frac{1}{(1/v)^*}(t) f^*(t) \right]^p \, dt \right)^{1/p} \nonumber \\
&\leq C \left( \int\limits_{\mathbb{H}^n} |f(z,t)|^p v(z,t)^p \, dz \, dt \right)^{1/p}.
\end{align}
This completes the proof of Pitt's inequality in the present case.  
For the case when \((p,q) = (2,2)\), we follow the line of reasoning presented in \cite[Theorem 3.1]{Tikhonov-weighted_norm_integral_transforms}. By a direct adaptation of the argument used in the proof of \cite[Theorem 2.3]{Tikhonov-weighted_norm_integral_transforms}, for \(s > 0\), the following inequality holds
\[
\left( \int\limits_{0}^{s} \tilde{f}^\ast(t)^2 dt \right)^{1/2} \le C \left( \int\limits_{0}^{x} \left( \int\limits_{0}^{1/t} f^\ast(s) ds \right)^2 dt \right)^{1/2}.
\]
Using the above inequality, we can deduce that: 
\begin{equation}\label{apply dec rearr}
\int\limits_{0}^{\infty} \tilde{f}^\ast(t)^2 U(t)^2 dt \le C \int\limits_{0}^{\infty} \left( \int\limits_{0}^{1/t} f^\ast(s) ds \right)^2 U(t)^2 dt.
\end{equation}
By applying a change of variables, the integral on the right-hand side becomes:
\[
\int\limits_{0}^{\infty} \left( \int\limits_{0}^{t} f^\ast(s) ds \right)^2 U(1/t)^2 t^{-2} dt.
\]
Next, by invoking the weighted Hardy inequality (Theorem~\ref{weighted hardy}~(a)), we conclude that 
\[
\int\limits_{0}^{\infty} \left( \int\limits_{0}^{t} f^\ast(s) ds \right)^2 U(1/t)^2 t^{-2} dt \leq C \int\limits_{0}^{\infty} \left( V(t) f^\ast(t) \right)^2 dt,
\]
if and only if 
\[
\sup_{s > 0} \left( \int\limits_{s}^{\infty} U(1/t)^2 t^{-2} dt \right)^{1/2} \left( \int\limits_{0}^{s} V(t)^{-2} dt \right)^{1/2} < \infty.
\]
This condition can be reformulated as:
\[
\sup_{s > 0} \left( \int\limits_{0}^{1/s} U(t)^2 dt \right)^{1/2} \left( \int\limits_{0}^{s} V(t)^{-2} dt \right)^{1/2} < \infty.
\]
Finally, by utilizing \eqref{apply dec rearr} and the methods applied in the previous case, we can derive the required inequality.
\subsection{An uncertainty principle} 

The uncertainty principle asserts that a function and its Fourier transform cannot both be simultaneously localized. 
A classical inequality that encapsulates this phenomenon is the Heisenberg–Pauli–Weyl inequality in \(\mathbb{R}^n\). 
In \cite{Cowling-Price-HPW-inequality}, inequalities of the type  
\begin{equation}\label{an uncertainty principle in Rn}
    \lVert f \rVert_{2}^2 \leq C \, \lVert u \, \widehat{f} \rVert_{q} \, \lVert v f \rVert_{p},
\end{equation}
for suitable choices of the weight functions \(v\) and \(u\), were studied.

Using Pitt's inequality on \(\mathbb{H}^n\), we establish an analogue of the inequality \eqref{an uncertainty principle in Rn} for the class of \(C_0^\infty(\mathbb{H}^n)\) functions. This result is motivated from \cite[Theorem 6.1]{Pitts_inequalities_Carli-Tikhonov}.
\begin{thm}
    Let \(1 \leq p,q \leq \infty\) and \(u\) and \(v\) be weights such that Pitt's inequality \eqref{pitts inequality Hn} holds. Then we have 
\begin{align}
\int\limits_{\mathbb{H}^n}\left(f(z,t)\right)^2dz\;dt\leq C \left( \int\limits_{\Omega \times \mathbb{C}^n}\left|\lambda u^{-1}\left(a,w\right)\tilde{f}\left(a,w \right)\right|^{q'}d\nu_2(a)\;dw\right)^{1/q'} \left( \int\limits_{\mathbb{H}^n}\left|tv(z,t)f(z,t)\right|^pdz\;dt\right)^{1/p}.
\end{align}
\end{thm}

\begin{proof}
    For \(f \in C_0^\infty(\mathbb{H}^n)\), integration by parts yields 
    \begin{align*}
        \int\limits_{-\infty}^{\infty}|f(z,t)|^2dt&=0-\int\limits_{-\infty}^{\infty}t \frac{\partial}{\partial t}|f(z,t)|^2\\
        &=-2\operatorname{Re}\left (\int\limits_{-\infty}^{\infty}t \overline{f(z,t)} \frac{\partial}{\partial t}f(z,t)dt\right).
    \end{align*}
    Integrating with respect to \(z\), we obtain
    \begin{align}
       \int\limits_{\mathbb{H}^n}|f(z,t)|^2 dz\;dt &=-2\operatorname{Re}\left (\int\limits_{\mathbb{C}^n \times \mathbb{R}}t \overline{f(z,t)} \frac{\partial}{\partial t}f(z,t)dt\right) \nonumber \\
       &=-2\operatorname{Re}\left \langle \frac{\partial f}{\partial t}, tf \right \rangle \nonumber \\
       &=-2\operatorname{Re}\left \langle \widetilde{\frac{\partial f}{\partial t}}, \widetilde{t f\,} \right \rangle. \label{L2 norm application}
    \end{align}
    The final equality follows from the Plancherel theorem.
    
    \noindent Here we recall that the function \(e_a\) is an eigenfunction of \(T=\frac{\partial}{\partial t}\) with eigen value \(i \la\). Using this, it is easy to very that 
\begin{equation}\label{transform of partial_t}
  \widetilde{\frac{\partial f}{\partial t}}(a,w)=-i \la \tilde{f}(a,w). 
\end{equation}
Substituting \eqref{transform of partial_t} into \eqref{L2 norm application} and applying Pitt's inequality, we obtain  
\begin{align}
    \int\limits_{\mathbb{H}^n} |f(z,t)|^2 \, dz \, dt
    &= 2 \, \operatorname{Re} \left( i \int\limits_{\Omega \times \mathbb{C}^n} 
        \lambda \, \tilde{f}(a,w) \, \overline{\widetilde{t f}(a,w)} \, dw \, d\nu_2(a) \right) \nonumber \\
    &= 2 \, \operatorname{Re} \left( i \int\limits_{\Omega \times \mathbb{C}^n} 
        \bigl(\lambda u^{-1} \tilde{f}(a,w)\bigr) \, 
        \overline{u \, (\widetilde{t f})(a,w)} \, dz \, d\nu_2 \right) \nonumber \\
    &\leq C \, \lVert \lambda u^{-1} \tilde{f} \rVert_{q'} \, 
            \lVert u (\widetilde{t f}) \rVert_{q} \nonumber \\
    &\leq C \, \lVert \lambda u^{-1} \tilde{f} \rVert_{q'} \, 
            \lVert v t f \rVert_{p}.
\end{align}
\noindent This completes the proof.

\end{proof}

\section{Necessary Condition for Pitt's inequality}\label{sec 4}

In this section, we establish the necessary conditions for the validity of Pitt’s inequality 
\eqref{pitts inequality Hn} in the setting of general radial weights.

In case of \(\mathbb{H}^n\), a function \(f(z,t)\) is said to be radial if it is radial in the first component. First, we recall that the normalized Strichartz Fourier transform of a radial function \(f\) on \(\mathbb{H}^n\) is given by 
\begin{equation}
    \tilde{f}(a,w)=\frac{k!(n-1)!}{(k+n-1)!}R^{n-1}_k(-\lambda,f)e^{n-1}_{k,\lambda}(w,0),
\end{equation}
where \(a=(\la,(2k+n)|\la|)\) and 
\begin{equation}\label{kth lag coeff}
    R^{n-1}_{k}(-\la,f)=\frac{k!(n-1)!}{(k+n-1)!}\int\limits_{\mathbb{C}^n}f^{-\la}(z)\varphi^{n-1}_{k,-\la}(z)dz=\frac{k!(n-1)!}{(k+n-1)!}\int\limits_{\mathbb{C}^n}f^{-\la}(z)\varphi^{n-1}_{k,\la}(z)dz.
\end{equation}


We begin by recalling some results from \cite[(9.5.2), (22.16.8)]{handbook} concerning the zeros of Bessel functions 
and Laguerre polynomials \(L^{n-1}_k\).


Let \(k \geq 1\) and \(\nu \geq 0\). Denote by \(J_\nu\) the Bessel function of the first kind, and let \(j_{\nu,s}\) be the \(s^{\text{th}}\) positive zero of \(J_\nu\). These zeros satisfy the interlacing property
\[
j_{\nu,1} < j_{\nu+1,1} < j_{\nu,2} < j_{\nu+1,2} < j_{\nu,3} < \cdots.
\]
For the Laguerre polynomials \(L^{n-1}_k\), the first positive zero \(x_1\) satisfies
\[
x_1 > \frac{j_{n-1,1}^2}{4 M_{k,n}}, \quad \text{where } M_{k,n} = k + \frac{n}{2}.
\]
Moreover, the derivative of \(L^{n-1}_k\) is given by
\begin{equation}\label{derivative of lag poly}
\frac{d}{dx} L^{n-1}_k(x) = -L^n_{k-1}(x), \quad k \geq 1.
\end{equation}


\noindent Observe that \(L^n_{k-1}(x)\) is positive on the interval 
\[
\left[0, \frac{j_{n,1}^2}{4 M_{k-1,n+1}}\right], \quad \text{where } M_{k-1,n+1} = k-1 + \frac{n+1}{2}.
\] 
This follows from the fact that the first positive zero of \(L^n_{k-1}(x)\) exceeds \(\frac{j_{n,1}^2}{4 M_{k-1,n+1}}\). 

\noindent Moreover, a straightforward comparison shows that
\[
\frac{j_{n-1,1}^2}{4 M_{k,n}} 
= \frac{j_{n-1,1}^2}{4 \left(k + n/2\right)}
< \frac{j_{n-1,1}^2}{4 \left(k-1 + (n+1)/2\right)}
< \frac{j_{n,1}^2}{4 \left(k-1 + (n+1)/2\right)}
= \frac{j_{n,1}^2}{4 M_{k-1,n+1}}.
\]

\noindent Consequently, \(L^n_{k-1}(x)\) is positive on 
\[
\left[0, \frac{j_{n-1,1}^2}{4 M_{k,n}}\right].
\] 
Combining this with \eqref{derivative of lag poly}, it follows that \(L^{n-1}_k(x)\) is decreasing on the interval 
\[
\left[0, \frac{j_{n-1,1}^2}{4 M_{k,n}}\right].
\]
We are now ready to derive the necessary condition on a radial weight for Pitt's inequality to hold. 
Assume that \(u\) and \(v^{-p'}\) are locally integrable.


\begin{theorem}\label{thm-necessary conditions for radial functions}
Let \(u\) and \(v\) be functions defined on \(\Omega \times \mathbb{C}^n\) and \(\mathbb{H}^n\), respectively, with \(v\) radial, and suppose that Pitt's inequality \eqref{pitts inequality Hn} holds for \(1<p,q<\infty\). Then the following necessary conditions hold:

\begin{itemize}
    \item For each \(k \geq 1\),
    \begin{equation}\label{necessary conditions k>=1}
    \sup_{s>0} 
    \left( \int\limits_{R_{k,s}} \int\limits_{|w|^2 < 2 \beta_k s} u\big(\lambda, (2k+n)|\lambda|, w\big)^q |\lambda|^{2n} \, dw \, d\lambda \right)^{1/q}
    \left( \int\limits_{|z|^2 < 2 \beta_k s} \int\limits_{|t| < s} v(z,t)^{-p'} \, dz \, dt \right)^{1/p'} < \infty,
    \end{equation}
    where 
    \(\beta_k = \frac{j_{n-1,1}^2}{4 M_{k,n}}\) and \(R_{k,s} = \{ (\lambda, (2k+n)|\lambda|) : |\lambda| < 1/s \}\).
    
    \item For \(k = 0\),
    \begin{equation}\label{necessary conditions k=0}
    \sup_{s>0} 
    \left( \int\limits_{R_{0,s}} \int\limits_{|w|^2 < 2 s} u\big(\lambda, n|\lambda|, w\big)^q |\lambda|^{2n} \, dw \, d\lambda \right)^{1/q}
    \left( \int\limits_{|z|^2 < 2 s} \int\limits_{|t| < s} v(z,t)^{-p'} \, dz \, dt \right)^{1/p'} < \infty.
    \end{equation}
\end{itemize}
\end{theorem}

\begin{proof}
First, fix \(k \geq 1\). For any \(s > 0\), consider the radial function
\begin{equation}\label{necessary proof func def}
f_k(z,t) =
\begin{cases}
v(z,t)^{-p'}, & \text{if } |z|^2 < 2 \beta_k s \text{ and } |t| < s,\\[2mm]
0, & \text{otherwise}.
\end{cases}
\end{equation}
Since \(v(z,t)\) is radial, it follows that \(f_k(z,t)\) is radial on \(\mathbb{H}^n\).

Observe that if \(|\lambda| < 1/s\) and \(|z|^2 < 2 \beta_k s\), then
\[
\frac{1}{2} |\lambda| |z|^2 < \frac{1}{2 s} \cdot 2 \beta_k s = \beta_k.
\]
Hence, using the decreasing property of \(L^{n-1}_k(x)\) established earlier, we have
\begin{equation}\label{necessary proof eq 1}
L^{n-1}_k\Big(\frac{1}{2} |\lambda| |z|^2\Big) > L^{n-1}_k(\beta_k).
\end{equation}

\noindent Similarly, we also have
\begin{equation}\label{necessary proof eq 2}
e^{-\frac{1}{4} |\lambda| |z|^2} > e^{-\beta_k / 2}.
\end{equation}
Combining \eqref{necessary proof eq 1} and \eqref{necessary proof eq 2}, we obtain 
\begin{equation}\label{bound for lag func}
\varphi^{n-1}_{k,\la}(z)=L^{n-1}_k\left( \frac{1}{2}|\la||z|^2\right)e^{-\frac{1}{4}|\la||z|^2}>L^{n-1}_k(\beta_k)e^{-\frac{\beta_k}{2}}=B_k(\text{ say}).
\end{equation}
  Consequently, we deduce that if \(|\la|<1/s\) and \(|z|^2< 2\beta_k s\), then \(\varphi^{n-1}_{k,\la}(z)\) is bounded below by a constant \(B_k\). 
  For the choice of \(f\) as in \eqref{necessary proof func def}, we have 
  $$
  R^{n-1}_k(-\la,f)=c_{n,k}\int\limits_{|z|^2<2\beta_k s}\int\limits_{|t|<s} v(z,t)^{-p'}e^{-i\la t}\varphi^{n-1}_{k,\la}(z)dt\;dz,
  $$
  where \(c_{n,k}=\frac{k!(n-1)!}{(k+n-1)!}\).
  
\noindent Next, taking the real part gives
 \begin{align}
\operatorname{Re}\big(R^{n-1}_k(-\lambda, f)\big) 
&= c_{n,k} \int\limits_{|z|^2 < 2 \beta_k s} \int\limits_{|t| < s} v(z,t)^{-p'} \cos(\lambda t) \, \varphi^{n-1}_{k,\lambda}(z) \, dt \, dz \nonumber \\
&\geq c_{n,k} \cos(1) \int\limits_{|z|^2 < 2 \beta_k s} \int\limits_{|t| < s} v(z,t)^{-p'} \, \varphi^{n-1}_{k,\lambda}(z) \, dt \, dz \nonumber \\
&\geq c_{n,k} \cos(1) \, B_k \int\limits_{|z|^2 < 2 \beta_k s} \int\limits_{|t| < s} v(z,t)^{-p'} \, dt \, dz \nonumber \\
&= B'_k \int\limits_{|z|^2 < 2 \beta_k s} \int\limits_{|t| < s} v(z,t)^{-p'} \, dt \, dz, \label{bound for real part}
\end{align}
where \(B'_k = c_{n,k} \cos(1) \, B_k\).

Let 
\[
R_{k,s} = \big\{ (\lambda, (2k+n)|\lambda|) : |\lambda| < 1/s \big\}.
\] 
Then we have
\begin{align}
&\Bigg( \int\limits_{R_{k,s}} \int\limits_{|w|^2 < 2 \beta_k s} 
\big| \tilde{f}(\lambda, (2k+n)|\lambda|, w) \big|^q 
\, u(\lambda, (2k+n)|\lambda|, w)^q \, |\lambda|^{2n} \, dw \, d\lambda \Bigg)^{1/q} \nonumber \\
&= c_{n,k} \Bigg( \int\limits_{R_{k,s}} \int\limits_{|w|^2 < 2 \beta_k s} 
u(\lambda, (2k+n)|\lambda|, w)^q \, |R^{n-1}_k(-\lambda, f)|^q \, |e^{n-1}_{k,\lambda}(w)|^q \, |\lambda|^{2n} \, dw \, d\lambda \Bigg)^{1/q} \nonumber \\
&= c_{n,k} \Bigg( \int\limits_{R_{k,s}} \int\limits_{|w|^2 < 2 \beta_k s} 
u(\lambda, (2k+n)|\lambda|, w)^q \, |R^{n-1}_k(-\lambda, f)|^q \, |\varphi^{n-1}_{k,\lambda}(w)|^q \, |\lambda|^{2n} \, dw \, d\lambda \Bigg)^{1/q} \nonumber \\
&> C_k \Bigg( \int\limits_{R_{k,s}} \int\limits_{|w|^2 < 2 \beta_k s} 
u(\lambda, (2k+n)|\lambda|, w)^q 
\Big( \int\limits_{|z|^2 < 2 \beta_k s} \int\limits_{|t| < s} v(z,t)^{-p'} \, dt \, dz \Big)^q 
\, |\lambda|^{2n} \, dw \, d\lambda \Bigg)^{1/q}, \label{estimate for integral in R_ks}
\end{align}
where the constant \(C_k\) is independent of the choice of \(s > 0\).
  
Applying Pitt's inequality \eqref{pitts inequality Hn} together with \eqref{estimate for integral in R_ks}, we obtain
\begin{align}
&C_k \Bigg( \int\limits_{R_{k,s}} \int\limits_{|w|^2 < 2 \beta_k s} 
u(\lambda, (2k+n)|\lambda|, w)^q 
\Big( \int\limits_{|z|^2 < 2 \beta_k s} \int\limits_{|t| < s} v(z,t)^{-p'} \, dt \, dz \Big)^q 
\, |\lambda|^{2n} \, dw \, d\lambda \Bigg)^{1/q} \nonumber \\
&\leq C \Bigg( \int\limits_{|z|^2 < 2 \beta_k s} \int\limits_{|t| < s} v(z,t)^{-p'p} \, v(z,t)^p \, dz \, dt \Bigg)^{1/p} \nonumber \\
&= C \Bigg( \int\limits_{|z|^2 < 2 \beta_k s} \int\limits_{|t| < s} v(z,t)^{-p'} \, dz \, dt \Bigg)^{1/p}.
\end{align}




The above inequality can be rewritten as
\begin{align}
\Bigg( \int\limits_{R_{k,s}} \int\limits_{|w|^2 < 2 \beta_k s} u(\lambda, (2k+n)|\lambda|, w)^q \, |\lambda|^{2n} \, dw \, d\lambda \Bigg)^{1/q} 
\Bigg( \int\limits_{|z|^2 < 2 \beta_k s} \int\limits_{|t| < s} v(z,t)^{-p'} \, dt \, dz \Bigg)^{1/p'} 
< C'_k.
\end{align}

Taking the supremum over \(s > 0\) then yields the required necessary conditions on the weights.

The case \(k = 0\) is similar, where we define
\begin{equation}
f_0(z,t) = 
\begin{cases}
v(z,t)^{-p'}, & \text{if } |z|^2 < 2 s \text{ and } |t| < s,\\[1mm]
0, & \text{otherwise}.
\end{cases}
\end{equation}

Now we use that \(L^{n-1}_0=1\) and for \(|\la|<1/s, |z|^2<2s\),
$$
\varphi^{n-1}_0(z)=L^{n-1}_0\left(\frac{1}{2} |\la||z|^2 \right)e^{-\frac{1}{4}|\la||z|^2}=e^{-\frac{1}{4}|\la||z|^2}>e^{-\frac{1}{2}}.
$$
\end{proof}

\section{Pitt's inequality for polynomial type weights}\label{sec 5}
In this section, we examine the condition \eqref{pitts cond1} for polynomial-type weights, thereby providing examples of non-trivial weights that satisfy Pitt's inequality. In this section, we use \(X \asymp Y\) to denote that \(X=CY\) for some constant \(C>0\) and \(X \lesssim Y\) to denote \(X\leq CY\) for some constant \(C>0\).

\subsection{Weight on the function side}
For \(\al>0\), let us consider the following weight
\begin{equation}\label{weight on function side}
    v_{\alpha}(z,t)=\left(|z|^4+16t^2\right)^{\alpha/4}.
\end{equation}
To find the decreasing rearrangement of the function \(1/v_{\alpha}\), we first calculate its distribution function. 
For \(M>0\),
\begin{align*}
    d_{1/v_{\alpha}}(M)&=\left|\left\{ (z,t):\left(|z|^4+16t^2\right)^{\alpha/4}>M \right \} \right| \nonumber \\
    &= \left|\left \{ (z,t):|z|^4+16t^2<M^{-4/\alpha} \right \} \right|.
\end{align*}

Here, \(|\cdot|\) denotes the measure of a set with respect to the measure \(dz \, dt\) on \(\mathbb{H}^n\).  
Thus, we have
\begin{equation*}
    d_{1/v_{\alpha}}(M)\asymp M^{-\frac{2(n+1)}{\alpha}}.
\end{equation*}

Then, the decreasing rearrangement of \(1/v_{\alpha}\) is given by 
\begin{align}
    (1/v_{\alpha})^*(t)&=\inf\left\{ M>0: d_{1/v}(M) \leq t \right\}\nonumber\\
    &\asymp t^{-\frac{\alpha}{2(n+1)}}\label{estimate of (1/v)*}.
\end{align}

\subsection{Weight on the Fourier transform side}
We consider two weights on the Fourier transform side. For \(\sigma,\rho>0\), consider the following two functions on \(\Omega \times \mathbb{C}^n\) given by
\begin{enumerate}
    \item \begin{equation}\label{second weight on transform side}
        u_{\rho}(a,w)=u\left(\la,(2k+n)|\la|,w\right)=\begin{cases}
        \left(|a|+|w|\right)^{-\rho}  & \text{ for } w \neq 0, \lambda \neq 0,\\
        0 & \text{ otherwise},
    \end{cases} 
    \end{equation}
    where \(|a|=(2k+n+1)|\lambda|\).
    \item \begin{equation}\label{weight on transform side}
    u_{\sigma,\rho}(a,w)=u\left(\la,(2k+n)|\la|,w\right)=\begin{cases}
        \left((2k+n)|\la|\right)^{-\sigma} |w|^{-\rho} & \text{ for } w \neq 0, |\lambda| > 1,\\
        0 & \text{ otherwise}.
    \end{cases}
\end{equation}
\end{enumerate}

We need to find the decreasing rearrangement of \(u_{\rho}\) and \(u_{\sigma,\rho}\). We denote \(| \cdot |\) as the measure of subset of \(\Omega \times \mathbb{C}^n\) with respect to measure \(d\nu_2\;dw\). 
For \(M>0\), observe that 
$$
\{ (a,w): u_\rho(a,w) > M \} 
\subseteq 
 \{ (a,w): |w| < M^{-1/\rho},  (2k+n)|\lambda| < M^{-1/\rho} \}.
$$
Using this we get 
\begin{align*}
    d_{u_{\rho}}(M)&=\left| \left\{( (\lambda,(2k+n)|\lambda|),w): u_{\rho}(a,w)>M \right \} \right|\\
    &\leq \left| \left\{ ( (\lambda,(2k+n)|\lambda|),w): |w| < M^{-1/\rho},  (2k+n)|\lambda| < M^{-1/\rho} \right\} \right|\\
    &\asymp  M^{-(4n+1)/\rho}.
\end{align*}
\noindent By definition, the decreasing rearrangement of \(u_{\rho}\) is given by 
\begin{align}
    u^\ast_{\rho}(t)&=\inf\{M>0: d_{u_\rho}(M) \leq t\}\nonumber \\
    &\lesssim t^{-\rho/(4n+1)}. \label{estimate of u*}
\end{align}
We also find the decreasing rearrangement for the function \(u_{\sigma,\rho}\).
For \(M>0\),
\begin{align*}
    d_{u_{\sigma,\rho}}(M)&=\left| \left\{( (\la,(2k+n)|\la|),w): \left((2k+n)|\la|\right)^{-\sigma }|w|^{-\rho}>M, |\lambda| > 1 \right \} \right|\\
    &=\left| \left \{ ( (\la,(2k+n)|\la|),w): |w|^{\rho}<M^{-1}\left((2k+n)|\la|\right)^{-\sigma }, |\la|>1\right\} \right|\\
    &\asymp M^{-2n/\rho}\int\limits_{|\lambda|>1} \sum\limits_{k} \left ( \frac{(k+n-1)!}{k!(n-1)!}\right )^2 (2k+n)^{-2n\sigma / \rho} |\lambda|^{-2n\sigma/ \rho}|\lambda|^{2n} d\lambda\\
    &\asymp M^{-2n/\rho},
\end{align*}
provided \(\sigma>\frac{\rho(2n+1)}{2n}\).

\noindent By definition, the decreasing rearrangement of \(u_{\sigma,\rho}\) is given by 
\begin{align}
    {u^\ast_{\sigma,\rho}}(t)&=\inf\{M>0: d_{u_{\sigma,\rho}}(M) \leq t\}\nonumber \\
    &\asymp t^{-\rho /2n}\label{estimate of u'*}.
\end{align}
when \(\sigma>\frac{\rho(2n+1)}{2n}\).
\subsection{Sufficient conditions on weights}
Using decreasing rearrangements of weights, we now determine conditions on weights using Theorem \ref{theorem-Pitts for Hn}, so that Pitt's inequality \eqref{pitts inequality Hn} holds.
We consider two pairs of weights \(v_{\alpha}\), \(u_{\rho}\) and \(v_{\alpha}\), \(u_{\sigma,\rho}\).

\begin{thm}
     Let \(1 < p \leq q < \infty\). Then the inequality \eqref{pitts inequality Hn} is satisfied for the weights \(u_{\rho}\) and \(v_{\alpha}\) if the following hold
        \begin{equation*}
            0<\alpha,\; 0<\rho,\; \alpha p'<2(n+1),\; q\rho<4n+1,\; 
        \end{equation*}
        and 
    \begin{equation*}
           - \frac{\alpha}{2(n+1)}+\frac{1}{p'}+\frac{\rho}{4n+1}-\frac{1}{q}=0.   
        \end{equation*}
\end{thm}
\begin{proof}
Let us denote the decreasing rearrangement of \(u_{\rho}\) by \(U\) and the decreasing rearrangement of \(1/v_{\alpha}\) by \(1/V\).

\noindent We find the conditions on \(\alpha,\rho\) such that \eqref{pitts cond1} holds, that is, 
\begin{equation*}
        \sup_{s>0}\left ( \int\limits_{0}^{1/s}U(t)^qdt \right )^{1/q} \left ( \int\limits_{0}^{s}V(t)^{-p'}dt \right )^{1/p'} <\infty.
    \end{equation*}
Using \eqref{estimate of (1/v)*}, we have 
\begin{align}
    \left(\int\limits_{0}^sV(t)^{-p'}dt\right)^{1/p'}\asymp & \left(\int\limits_{0}^s  t^{-\frac{\alpha p'}{2(n+1)}} dt\right)^{1/p'}\nonumber\\
 &\asymp s^{-\frac{\alpha}{2(n+1)}+\frac{1}{p'}},\nonumber
    \end{align}
provided 
\begin{equation}\label{weights cond1}
    \frac{\alpha p'}{2(n+1)}<1.
\end{equation}
Similarly, using \eqref{estimate of u*}, we have
\begin{align}
   \left( \int\limits_{0}^{1/s}U(t)^qdt\right)^{1/q}& \lesssim \left( \int\limits_{0}^{1/s}t^{-\frac{q\rho}{4n+1}}dt\right)^{1/q}\nonumber\\
   &\asymp s^{\frac{\rho}{4n+1}-\frac{1}{q}},\nonumber,
\end{align}
provided
\begin{equation}\label{weights cond2}
    \frac{q\rho}{4n+1}<1.
\end{equation}
Thus we deduce, 
\begin{align*}
     \sup_{s>0}\left ( \int\limits_{0}^{1/s}U(t)^qdt \right )^{1/q} \left ( \int\limits_{0}^{s}V(t)^{-p'}dt \right )^{1/p'} &\lesssim\sup_{s>0} s^{-\frac{\alpha}{2(n+1)}+\frac{1}{p'}+\frac{\rho}{4n+1}-\frac{1}{q}}.
\end{align*}
The above is finite if and only if 
\begin{equation}\label{weights cond3}
   - \frac{\alpha}{2(n+1)}+\frac{1}{p'}+\frac{\rho}{4n+1}-\frac{1}{q}=0.
\end{equation}
This completes the proof.
\end{proof}
The next theorem provides sufficient conditions for the Pitt's inequality \eqref{pitts inequality Hn} to hold for weights \(v_{\alpha}\) and \(u_{\sigma,\rho}\).
\begin{thm}
    Let \(1 < p \leq q < \infty\). Then the inequality \eqref{pitts inequality Hn} is satisfied for the weights \(u_{\sigma,\rho}\) and \(v_{\alpha}\) if the following hold
        \begin{equation*}
            0<\alpha,\; 0<\sigma ,\; 0<\rho,\; \rho(2n+1)<2n\sigma,\; \alpha p'<2(n+1),\; q\rho<2n,\; 
        \end{equation*}
        and 
    \begin{equation*}
           -\frac{\alpha}{2(n+1)}+\frac{1}{p'}+\frac{\rho}{2n}-\frac{1}{q}=0.  
        \end{equation*}
\end{thm}
\begin{proof}
    Proceeding on the same lines as previous theorem, we denote the decreasing rearrangement of \(u_{\sigma,\rho}\) by \(U\) and the decreasing rearrangement of \(1/v_{\alpha}\) by \(1/V\).
    We have the estimate
\begin{equation*}
    \left(\int\limits_{0}^sV(t)^{-p'}dt\right)^{1/p'}\asymp s^{-\frac{\alpha}{2(n+1)}+\frac{1}{p'}},
    \end{equation*}
provided 
\begin{equation}\label{second weights cond1}
    \frac{\alpha p'}{2(n+1)}<1.
\end{equation}
Similarly using decreasing rearrangement of \(u_{\sigma,\rho}\), we have 
\begin{align}
   \left( \int\limits_{0}^{1/s}U(t)^qdt\right)^{1/q}& \asymp \left( \int\limits_{0}^{1/s}t^{-\frac{q\rho}{2n}}dt\right)^{1/q}\nonumber\\
   &\asymp s^{\frac{\rho}{2n}-\frac{1}{q}},\nonumber
\end{align}
provided
\begin{equation}\label{second weights cond2}
    \frac{q\rho}{2n}<1.
\end{equation}
\noindent Therefore, we have 
\begin{equation*}
        \sup_{s>0}\left ( \int\limits_{0}^{1/s}U(t)^qdt \right )^{1/q} \left ( \int\limits_{0}^{s}V(t)^{-p'}dt \right )^{1/p'} \asymp \sup_{s>0}s^{-\frac{\alpha}{2(n+1)}+\frac{1}{p'}+\frac{\rho}{2n}-\frac{1}{q}}.
    \end{equation*}
The above supremum is finite if and only if
\begin{equation}\label{weights cond3}
   - \frac{\alpha}{2(n+1)}+\frac{1}{p'}+\frac{\rho}{2n}-\frac{1}{q}=0.
\end{equation}
\end{proof}


\subsection{Necessary conditions on weights} In this subsection, we derive necessary conditions on the weights using Theorem \ref{thm-necessary conditions for radial functions}. 
Note that the weight defined in \eqref{weight on function side} is radial on \(\mathbb{H}^n\). As in previous subsection, we first consider the pair of weights \(v_{\alpha}\), \(u_{\rho}\).
\begin{thm}\label{thm 1 necessary conditions}
   Assume that for the weights \(v_{\alpha}\), \(u_{\rho}\),  Pitt's inequality \eqref{pitts inequality Hn} holds for \(1<p,q<\infty\). Then the following conditions necessarily hold :
$$0<\alpha,\,0<\rho,\,\alpha p'<2(n+1),$$
and 
$$-\rho < -\frac{(n+1)}{q}-\frac{\alpha}{2}+\frac{(n+1)}{p'}<\frac{\rho}{2}.$$
   \end{thm}
\begin{proof}
    \ Let \(1<p,q<\infty\) be fixed. Now recall from Theorem \ref{thm-necessary conditions for radial functions}, the necessary conditions on functions \(u_{\rho}\) and \(v_{\alpha}\) state that for each \(k\geq 1\), 
     \begin{equation*}
\sup_{s>0}\left(\int\limits_{R_{k,s}}\int\limits_{|w|^2< 2\beta_k s} u_{\rho}((\la,(2k+n)|\la|,w)^q|\la|^{2n}dw\;d\la\right)^{1/q}\left( \int\limits_{|z|^2 < 2\beta_k s}\int\limits_{|t|<s}v_{\alpha}(z,t)^{-p'} dz\;dt\right)^{1/p'}< \infty,   
    \end{equation*}
    where \(\beta_k=\frac{j^2_{n-1,1}}{4M_k}\), \(R_{k,s}=\{(\la,(2k+n)|\la|): |\la|<1/s\}\) and for \(k=0\),
    \begin{equation*}
    \sup_{s>0}\left(\int\limits_{R_{0,s}}\int\limits_{|w|^2< 2 s} u_{\rho}((\la,n|\la|,w)^q|\la|^{2n}dw\;d\la\right)^{1/q}\left( \int\limits_{|z|^2 < 2 s}\int\limits_{|t|<s}v_{\alpha}(z,t)^{-p'} dz\;dt\right)^{1/p'}< \infty.   
    \end{equation*} 
  Let \(k\geq 1\).
To derive these conditions, we first compute the integrals appearing in the above expression.
We proceed by finding a lower bound for
\[I_{k,s}=\left(\int\limits_{R_{k,s}}\int\limits_{|w|^2< 2 \beta_k s} u_{\rho}((\lambda,(2k+n)|\lambda|,w)^q|\lambda|^{2n}dw\;d\lambda\right)^{1/q}.\]
\noindent For \(|\lambda|<1/s\) and \(|w|<\sqrt {2\beta_k s}\), one can see that 
$$
(u_\rho(a,w))^q=((2k+n+1)|\lambda|+|w|)^{-\rho q} \geq ((2k+n+1)/s + \sqrt{2 \beta_k s})^{-\rho q}
$$
Thus 
\begin{align*}
 I_{k,s} &\geq  \left( \frac{2k+n+1}{s}+\sqrt{2 \beta_k s} \right)^{-\rho} \left(\int\limits_{R_{k,s}}\int\limits_{|w|^2< 2 \beta_k s} |\lambda|^{2n}dw\;d\lambda\right)^{1/q}\\&\asymp\left( \frac{2k+n+1}{s}+\sqrt{2 \beta_k s} \right)^{-\rho}s^{-\frac{(n+1)}{q}}. 
\end{align*}
Moreover, we also have 
 \begin{align*}
     J_{k,s}&=\left( \int\limits_{|z|^2 < 2\beta_k s}\int\limits_{|t|<s}v_{\alpha}(z,t)^{-p'} dz\;dt\right)^{1/p'}\nonumber\\
     &=\left( \int\limits_{|z|^2 < 2\beta_k s}\int\limits_{|t|<s}(|z|^4+16t^2)^{-p'\alpha/4}dz\;dt\right)^{1/p'}\nonumber\\
     &\asymp s^{-\frac{\alpha }{2}+\frac{n+1}{p'}}.
 \end{align*}
 This holds provided \(\alpha p'<2(n+1)\).

 \noindent Combining the estimates for \(I_{k,s}\) and \(J_{k,s}\), we get
\begin{align*}
&\sup_{s>0}\left(\int\limits_{R_{k,s}}\int\limits_{|w|^2< 2\beta_k s} u_{\rho}((\lambda,(2k+n)|\lambda|,w)^q|\lambda|^{2n}dw\;d\lambda\right)^{1/q}\left( \int\limits_{|z|^2 < 2\beta_k s}\int\limits_{|t|<s}v_{\alpha}(z,t)^{-p'} dz\;dt\right)^{1/p'}\\
&\gtrsim \sup_{s>0} s^{-\frac{\alpha }{2}+\frac{n+1}{p'}}\left( \frac{2k+n+1}{s}+\sqrt{2 \beta_k s} \right)^{-\rho}s^{-\frac{(n+1)}{q}}
\end{align*}
The above is finite if and only if 
$$-\rho < -\frac{(n+1)}{q}-\frac{\alpha}{2}+\frac{(n+1)}{p'}<\frac{\rho}{2}.$$
\end{proof}
\begin{rmk}
    In deriving the necessary conditions in the preceding theorem, we used a lower bound for the integral. 
While this approach identifies conditions that must be satisfied, it does not capture all the necessary constraints that would follow from evaluating the integral explicitly.
\end{rmk}
\noindent The next theorem concerns the necessary conditions for the pair of weights \(v_{\alpha}\) and \(u_{\sigma,\rho}\).
\begin{thm}
   Suppose that for the weights \(u_{\sigma,\rho}\), \(v_{\alpha}\)  Pitt's inequality \eqref{pitts inequality Hn} holds for \(1<p,q<\infty\). Then the following conditions necessarily hold :
   \begin{equation*}
       0<\alpha,\; 0<\sigma ,\; 0<\rho,\;\alpha p'<2(1+n),\;q\rho<2n,\; 
   \end{equation*}    
  and 
   \begin{equation*}
\begin{aligned}
-\frac{\alpha}{2}+\frac{n+1}{p'}+\frac{2n-q\rho}{2q}+\frac{-2n+q\sigma-1}{q} 
&\ge 0, 
&& \text{if } q\sigma < 2n+1, \\
 -\frac{\alpha}{2}+\frac{n+1}{p'}+\frac{2n-q\rho}{2q}
&> 0, 
&& \text{if } q\sigma = 2n+1, \\
 -\frac{\alpha}{2}+\frac{n+1}{p'}+\frac{2n-q\rho}{2q}
&\ge 0, 
&& \text{if } q\sigma > 2n+1.
\end{aligned}
\end{equation*}
\end{thm}
\begin{proof}
   \ Let \(1<p,q<\infty\) be fixed.
   Let \(k\geq 1\).
To derive necessary conditions, we first compute the integrals appearing in \eqref{necessary conditions k>=1}.
\noindent For \(s>0\), let 
    \begin{align*}       I_{k,s}&=\left(\int\limits_{R_{k,s}}\int\limits_{|w|^2< 2 \beta_k s} u_{\sigma,\rho}((\la,n|\la|,w)^q|\la|^{2n}dw\;d\la\right)^{1/q}\\
        &\asymp \begin{cases}
            \left( \displaystyle\int\limits_{1<|\la|<1/s}\int\limits_{|w|^2<2\beta_k s}|w|^{-q\rho}|\la|^{-q\sigma}|\la|^{2n}dw\;d\la \right)^{1/q} & \text{ if } s<1,\\
            0 & \text{ if } s\geq 1
        \end{cases} \\
         &\asymp \begin{cases}
     \left(s^{(2n-q\rho)/2} \displaystyle\int\limits_{1<|\lambda|<1/s}|\lambda|^{2n-q\sigma}d\lambda \right)^{1/q} & \text{ if } s<1,\\
        0 & \text{ if } s \geq 1.
        \end{cases}
    \end{align*}
The above holds when \(q\rho<2n\). Thus we are left with the case when \(s<1\).
Furthermore, 
\begin{align*}
    \displaystyle\int\limits_{1<|\lambda|<1/s}|\lambda|^{2n-q\sigma}d\lambda&= \begin{cases}
        s^{-2n+q\sigma-1}-1 & \text{ if } 2n-q\sigma \neq -1,\\
        \operatorname{log(1/s)} & \text{ if } 2n-q\sigma=-1.
    \end{cases}
\end{align*}
Thus, for \(s<1\),
\begin{equation*}
    I_{k,s}\asymp\begin{cases}
        \left(s^{2n-q\rho)/2} \left(\frac{s^{-2n+q\sigma-1}-1}{2n-q\sigma+1}\right) \right)^{1/q} & \text{ if } 2n-q\sigma \neq -1,\\
        \left( s^{2n-q\rho)/2} \operatorname{log}(1/s)\right)^{1/q} & \text{ if } 2n-q\sigma=-1.
    \end{cases}
\end{equation*}

\noindent Using the proof of Theorem \ref{thm 1 necessary conditions}, we have
 \begin{align*}
     J_{k,s}&=\left( \int\limits_{|z|^2 < 2\beta_k s}\int\limits_{|t|<s}v_{\alpha}(z,t)^{-p'} dz\;dt\right)^{1/p'}\nonumber\\
     &\asymp s^{-\frac{\alpha }{2}+\frac{n+1}{p'}},
 \end{align*}
provided \(\alpha p'<2(n+1)\).

\noindent Again combining the estimates for \(I_{k,s}\) and \(J_{k,s}\), we deduce
\begin{align}
&\sup_{s > 0} 
\Bigg( \int\limits_{R_{k,s}} \int\limits_{|w|^2 < 2\beta_k s} u_{\sigma,\rho}(\lambda, (2k+n)|\lambda|, w)^q \, |\lambda|^{2n} \, dw \, d\lambda \Bigg)^{1/q} 
\Bigg( \int\limits_{|z|^2 < 2\beta_k s} \int\limits_{|t| < s} v_{\alpha}(z,t)^{-p'} \, dz \, dt \Bigg)^{1/p'} \nonumber \\
&= \sup_{0<s<1} 
\Bigg( \int\limits_{R_{k,s}} \int\limits_{|w|^2 < 2\beta_k s} u_{\sigma,\rho}(\lambda, (2k+n)|\lambda|, w)^q \, |\lambda|^{2n} \, dw \, d\lambda \Bigg)^{1/q} 
\Bigg( \int\limits_{|z|^2 < 2\beta_k s} \int\limits_{|t| < s} v_{\alpha}(z,t)^{-p'} \, dz \, dt \Bigg)^{1/p'} \nonumber\\
&\asymp \displaystyle\sup_{0<s<1} \begin{cases}
 s^{-\frac{\alpha}{2}+\frac{n+1}{p'}+\frac{2n-q\rho}{2q}} \left(\frac{s^{-2n+q\sigma-1}-1}{2n-q\sigma+1}\right)^{1/q} & \text{ if } 2n-q\sigma \neq 1,\\
 s^{-\frac{\alpha}{2}+\frac{n+1}{p'}+\frac{2n-q\rho}{2q}} \left(\operatorname{log}(1/s) \right)^{1/q} & \text { if } 2n-q\sigma =-1.\label{finiteness of supremum}
\end{cases}
\end{align}
To conclude the conditions that imply finiteness of the above quantity, we break our study in three cases.
\begin{enumerate}
    \item Case: \(2n-q\sigma>-1\)\\
    In this case, the supremum in \eqref{finiteness of supremum} is finite if and only if 
    \begin{equation}
        -\frac{\alpha}{2}+\frac{n+1}{p'}+\frac{2n-q\rho}{2q}+\frac{-2n+q\sigma-1}{q} \geq 0.
    \end{equation}
    \item Case: \(2n-q\sigma<-1\)\\
    In this case, the supremum in \eqref{finiteness of supremum} is finite if and only if 
    \begin{equation}
        -\frac{\alpha}{2}+\frac{n+1}{p'}+\frac{2n-q\rho}{2q} \geq 0.
    \end{equation}
    \item Case: \(2n-q\sigma=-1\)\\
    In this case, the supremum in \eqref{finiteness of supremum} is finite if and only if 
    \begin{equation}
        -\frac{\alpha}{2}+\frac{n+1}{p'}+\frac{2n-q\rho}{2q} > 0.
    \end{equation}
\end{enumerate}
\end{proof}
\begin{rmk}
    Even though we have obtained sufficient and necessary conditions for the above choice of weights, however, we have not got if and only if result as in the case of \({\mathbb{R}^n}\).
\end{rmk}
In the Euclidean case, if we let $ f_r(x) = f(rx), r>0 $ then $ \hat{f_r}(\xi) = r^{-n}\hat{f}(r^{-1}\xi).$  This allows us to get a necessary condition for the weighted norm inequality 
$$ \left( \int_{\mathbb R^n} |\hat{f}(\xi)|^q u(\xi) d\xi \right)^{1/q} \leq C \left( \int_{\mathbb R^n} |f(x)|^p v(x) dx \right)^{1/p}, $$
when $ u $ and $ v $ are both homogeneous under the above dilation. We can get prove a similar result for the Strichartz Fourier transform.

From the definition, it is easy to see that the Strichartz Fourier transform has the following property. Let $  \delta_rf(z,t) = f(rz, r^2t) $ be the non-isotropic dilation of the function $ f.$  Then it follows from the definition that
$ \widetilde{\delta_rf}(a,z) =  r^{-Q} \widetilde{f}(r^{-2} a, r z) =: r^{-Q} D_{r^{-1}} \widetilde{f}(a,z).$

\begin{thm} Let the weights $ u $ and $ v $ be such that $  \delta_r v(z,t) = r^{\alpha} v(z,t) $ and $ D_r u(a,z) = r^{\beta} u(a,z).$  A necessary condition for the validity of the inequality \eqref{pitts inequality Hn} is 
that  $$ \frac{1}{p} + \frac{1}{q} = 1-\frac{(\alpha+\beta)}{Q}.$$
\end{thm}
\begin{proof} Let \(u\) and \(v\) be weights with the condition given in hypothesis for which Pitt's inequality \eqref{pitts inequality Hn} holds. To establish the conditions on \(\alpha\) and \(\beta\), we replace \(f\) by \(\delta_rf\) in \eqref{pitts inequality Hn}. 
Using change of variables, it is easy to see that 
\begin{align}
  \left(  \int\limits_{\mathbb{H}^n} \left| \delta_rf(z,t)\right|^p v
  (z,t)^p\,dz\,dt\right)^{1/p}&=\left(  \int\limits_{\mathbb{H}^n} \left| f(rz,r^2t)\right|^p v
  (z,t)^p\,dz\,dt\right)^{1/p}\nonumber\\
  &=\left(  \int\limits_{\mathbb{H}^n} r^{-Q}\left| f(z,t)\right|^p \left(\delta_{r^{-1}}v
  (z,t)\right)^p\,dz\,dt\right)^{1/p}\nonumber\\
  &=\left(  \int\limits_{\mathbb{H}^n} r^{-Q-p\alpha}\left| f(z,t)\right|^p v
  (z,t)^p\,dz\,dt\right)^{1/p}\nonumber\\
  &=r^{-Q/p-\alpha}\left(  \int\limits_{\mathbb{H}^n} \left| f(z,t)\right|^p v
  (z,t)^p\,dz\,dt\right)^{1/p}\label{inequality after dilation on Hn}.
\end{align}
Furthermore, using the relation $ \widetilde{\delta_rf}(a,z) =  r^{-Q} \widetilde{f}(r^{-2} a, r z)$ and applying change of variables we obtain 
\begin{align}
    \left( \int\limits_{\Omega \times \mathbb{C}^n}\left|\widetilde{\delta_rf}(a,w)\right|^q u(a,w)^q dw\;d\nu_2(a)  \right )^{1/q}&=\left( \int\limits_{\Omega \times \mathbb{C}^n}r^{-qQ}\left|\tilde{f}(r^{-2}a,rw)\right|^q u(a,w)^q dw\;d\nu_2(a)  \right )^{1/q}\nonumber\\
    &=\left( \int\limits_{\Omega \times \mathbb{C}^n}r^{-qQ+Q}\left|\tilde{f}(a,w)\right|^q \left(D_ru(a,w)\right)^q dw\;d\nu_2(a)  \right )^{1/q}\nonumber\\
    &=\left( \int\limits_{\Omega \times \mathbb{C}^n}r^{-qQ+Q+q\beta}\left|\tilde{f}(a,w)\right|^q u(a,w)^q dw\;d\nu_2(a)  \right )^{1/q}\nonumber\\
    &=r^{Q/q-Q+\beta}\left( \int\limits_{\Omega \times \mathbb{C}^n}\left|\tilde{f}(a,w)\right|^q u(a,w)^q dw\;d\nu_2(a)  \right )^{1/q}\label{inequality after dilation on Omega x Cn}.
\end{align}
We know that the constant appearing in \eqref{pitts inequality Hn} is independent of choice of \(f\). Thus from \eqref{inequality after dilation on Hn} and \eqref{inequality after dilation on Omega x Cn}, we deduce that 
\begin{equation*}
    -\frac{Q}{p}-\alpha=\frac{Q}{q}-Q+\beta,
\end{equation*}
which implies the necessary condition 
\begin{equation*}
    \frac{1}{p}+\frac{1}{q}=1-\frac{(\alpha+\beta)}{Q}.
\end{equation*}
\end{proof}

\section{Paley's Inequality}\label{sec 6}
In the context of \(\mathbb{R}^n\), the Hausdorff-Young inequality and Pitt's inequality examine the \(L^p-L^{p'}\) and \(L^p-L^q\) boundedness of the Fourier transform respectively. By contrast, Paley's inequality addresses the boundedness of the Fourier transform from \(L^p\) to \(L^p_{\nu}\) for a specified measure \(\nu\). In 1960, H\"ormander \cite[Theorem 1.10]{Hormander-1960-paley}, established the following result :
\begin{theorem}\label{hormander-paley}
    Let \(\phi\) be a positive measurable function in \(L^{1,\infty}(\mathbb{R}^n)\).
    Then for \(1<p \leq 2\), we have 
   $$
   \left(\int\limits_{\mathbb{R}^n}|\widehat{f}(\xi)|^p\phi(\xi)^{2-p}d\xi\right)^{1/p}\leq C_p\lVert f \rVert_p.
   $$
\end{theorem}

Here, \(|\cdot|\) denotes the measure with respect to the Plancherel measure. 

We aim to establish a corresponding result for the Strichartz Fourier transform on the Heisenberg group, for which we employ the Marcinkiewicz interpolation theorem \cite[Theorem 4.13, Corollary 4.14]{bennett's_book}.

\begin{theorem}\label{Marcinkiewicz}
    Let \((X,\mu)\) and \((Y,\nu)\) be two measure spaces and \(1 \leq p_1 < p_2 < \infty\) and \(1 \leq q_1,q_2 \leq \infty, \; q_1 \neq q_2,\; p_i\leq q_i,\; \{i=1,2\}\). Suppose T is a quasilinear operator on \(L^{p_1,1}(X,\mu)+L^{p_2,1}(X,\mu)\) taking values in \((Y,\nu)\) such that \(T\) is simultaneously weak type \((p_1,q_1)\) and \((p_2,q_2)\), then for any \(t \in (0,1)\), and 
    $$
    \frac{1}{p}=\frac{1-t}{p_1}+\frac{t}{p_2},\; \frac{1}{q}=\frac{1-t}{q_1}+\frac{t}{q_2},
    $$
    the operator is bounded from \(L^p\) to \(L^q\).
\end{theorem}
We now state Paley's inequality for the Strichartz Fourier Transform.

\begin{theorem}
    Let \(\phi\) be a positive function on \(\Omega \times \mathbb{C}^n\) and suppose that \(\phi \in L^{1,\infty}(\Omega \times \mathbb{C}^n, dw\;d\nu_2)\). Then for \(1< p \leq 2\), there exists a constant \(C_p\) such that
    \begin{equation}
        \left( \int\limits_{\Omega \times \mathbb{C}^n}\left | \tilde{f}(a,w) \right|^p\phi^{2-p}(a,w)dw \; d\nu_2(a)\right)^{1/p}\leq C_p \lVert f \rVert_p.
    \end{equation}
\end{theorem}
\begin{proof}
    We prove this theorem along the same lines as in \(\mathbb{R}^n\). Note that for \(p=2\), the inequality directly follows from the Plancherel theorem. Now, fix \(1<p <2\) and consider the operator \( T\) on \(L^{1,1}(\mathbb{H}^n)+ L^{2,1}(\mathbb{H}^n) \), defined as 
    \begin{equation}
        T(f)=\frac{\tilde{f}}{\phi},
    \end{equation} which takes values in the space of measurable functions on the measure space \((\Omega \times \mathbb{C}^n; \phi^2(a,w) dw\; d\nu_2)\).
    
    \noindent Observe that for \(f \in L^2(\mathbb{H}^n)\), 
    \begin{align*}
        \int\limits_{\Omega \times \mathbb{C}^n}(T(f)(a,w))^2\phi^2(a,w)dw\;d\nu_2(a)&=\int\limits_{\Omega \times \mathbb{C}^n}\left(\tilde{f}(a,w)\right)^2 dw\;d\nu_2(a)\\
        &=\lVert f \rVert_2^2.
    \end{align*}
Hence, the operator T is of strong type \((2,2)\) and consequently also of weak type \((2,2)\).

   Next, we show that \(T\) is of weak type \((1,1)\), which is equivalent to proving that
    \begin{equation}\label{equiv weak 1,1}
    \left| \{(a,w): |Tf(a,w)|>s \} \right|_{\phi^2} \leq C'\frac{\lVert f \rVert_1}{s},\; s>0.
    \end{equation}
    
Here, \(|\cdot|_{\phi^2}\) denotes the weighted measure with respect to \(\phi^2(a,w)\). 
Note that it is assumed that \(\phi \in L^{1,\infty}(\Omega \times \mathbb{C}^n)\).
Therefore there exists a constant \(C\) such that 
\begin{equation}\label{phi-in-L1inf}
\left |\{(a,w): \phi(a,w)>s \}\right| \leq \frac{C}{s}.
\end{equation}
Using the estimate \eqref{L1 to L inf bdd}, we have
$$
\left|Tf(a,w)\right|=\left| \frac{\tilde{f}(a,w)}{\phi(a,w)}\right|\leq c_{n,1}\frac{\lVert f \rVert_1}{\phi(a,w)}.
$$
Hence, if 
$$
\left | Tf(a,w)\right|>s,
$$
then 
$$
\phi(a,w) \leq c_{n,1}\frac{\lVert f \rVert_1}{s}.
$$
Thus, to obtain the estimate \eqref{equiv weak 1,1}, we show that for \(\sigma >0\),
\begin{equation}\label{claim for weak type}
  \left | \{(a,w):\phi(a,w) \leq \sigma\} \right|_{\phi^2}  \leq 2K\sigma,
\end{equation}
where the constant \(K\) is independent of \(\sigma\).
By Cake Layer representation theorem in \cite{Lieb-Loss} and \eqref{phi-in-L1inf}, we obtain
\begin{align*}
  \left | \{(a,w):\phi(a,w) \leq \sigma\} \right|_{\phi^2}  &=\int\limits_{\phi(a,w) \leq \sigma}\phi^2(a,w)dw\;d\nu_2(a)\\
  &=\int\limits_{0}^\infty2t \left| \{(a,w): t< \phi \leq \sigma\}\right|dt\\
  &= \int\limits_0^\sigma 2t( m(t)-m(\sigma))dt,\; \text{here } m(t)=\left|\{(a,w): \phi(a,w)>t\} \right|\\
  &=\int\limits_0^\sigma 2t m(t)dt-\sigma^2m(\sigma)\\
  & \leq \int\limits_0^\sigma 2t m(t)dt \leq 2C\sigma .
  \end{align*}
This establishes the claim in \eqref{claim for weak type}. 
Consequently, the operator \(T\) is of weak type \((1,1)\). 
Applying the Marcinkiewicz interpolation theorem (Theorem \ref{Marcinkiewicz}) then yields the desired result.
\end{proof}
\section{ Hardy's inequality and Pitt's inequality}\label{sec 7}
In this section, we see the relation between Pitt's inequality and Hardy's inequality for some specific choice of weights. For \(f \in \mathcal{S}(\mathbb{R}^n)\), as discussed in the introduction the classical Pitt's inequality in \(\mathbb{R}^n\) states that
\begin{equation*}
    \left( \int\limits_{\mathbb{R}^n} |\xi|^{-\sigma q} |\widehat{f}(\xi)|^q \, d\xi \right)^{1/q} 
    \leq C_{\sigma} \left( \int\limits_{\mathbb{R}^n} |x|^{\alpha p} |f(x)|^p \, dx \right)^{1/p},    
\end{equation*}  
which holds if and only if  
\[
   0 \leq \sigma < \frac{n}{q}, 
   \qquad 0 \leq \alpha < \frac{n}{p'}, 
   \qquad \text{and} \qquad 
   \alpha - \sigma = n \left( 1 - \frac{1}{p} - \frac{1}{q} \right).
\]  
When \(p=q=2\), the inequality reduces to 
\begin{equation}\label{pitts p=2}
   \left( \int\limits_{\mathbb{R}^n} |\xi|^{-2\sigma } |\widehat{f}(\xi)|^2 \, d\xi \right)^{1/2} 
    \leq C_{\sigma} \left( \int\limits_{\mathbb{R}^n} |x|^{2\sigma } |f(x)|^2 \, dx \right)^{1/2} ,
\end{equation}
for \(0 \leq\sigma<n/2.\)
\noindent  Replacing \(f\) by its Fourier transform in the above equation, we obtain
\begin{equation}\label{pitts replace fourier}
   \left( \int\limits_{\mathbb{R}^n} |x|^{-2\sigma } |{f}(x)|^2 \, dx\right)^{1/2} 
    \leq C_{\sigma} \left( \int\limits_{\mathbb{R}^n} |\xi|^{2\sigma } |\widehat{f}(\xi)|^2 \, d\xi \right)^{1/2}.
\end{equation}
\noindent We can rewrite the above two inequalities in terms of fractional powers of the Laplacian. The inequality \eqref{pitts p=2} leads to 
\begin{equation}\label{pitts laplacian}
    \langle (-\Delta)^{-\sigma}f,f \rangle \leq C_{\sigma} \int\limits_{\mathbb{R}^n}|x|^{2\sigma}|f(x)|^2\,dx.
\end{equation}
In addition, the inequality \eqref{pitts replace fourier} reduces to 
\begin{equation}\label{hardy ineq}
\int\limits_{\mathbb{R}^n}|x|^{-2\sigma}|f(x)|^2\,dx \leq C_{\sigma} \langle(-\Delta)^{\sigma}f,f\rangle,
\end{equation}
which is the Hardy's inequality for fractional powers of the Laplacian.
By reversing the argument, we can deduce inequality Pitt's inequality via Hardy's inequality. We can summarize this as the following proposition.
\begin{proposition}
    For \(f \in \mathcal{S}(\mathbb{R}^n)\) and \(0 \leq \sigma <n/2\), the Pitt's inequality
    $$
    \langle (-\Delta)^{-\sigma}f,f \rangle \leq C_{\sigma} \int\limits_{\mathbb{R}^n}|x|^{2\sigma}|f(x)|^2\,dx.
    $$
    is equivalent to the Hardy's inequality 
    $$
    \int\limits_{\mathbb{R}^n}|x|^{-2\sigma}|f(x)|^2\,dx \leq C_{\sigma} \langle(-\Delta)^{\sigma}f,f\rangle.
    $$
\end{proposition}
\noindent This shows the connection between Pitt's inequality and the Hardy's inequality.

The inequality \eqref{pitts laplacian} can also be obtained from the inequality \eqref{hardy ineq} by the following argument.
\begin{equation*}
    \langle (-\Delta)^{\sigma}f,f \rangle^2\leq \left(\int\limits_{\mathbb{R}^n}|(-\Delta)^{\sigma}f(x)|^2|x|^{2\sigma}\,dx \right) \left( \int\limits_{\mathbb{R}^n}|x|^{-2\sigma}|f(x)|^2\,dx \right).
\end{equation*}
Now applying Hardy's inequality the above takes the form
\begin{equation*}
    \langle (-\Delta)^{\sigma}f,f \rangle\leq C_{\sigma} \left(\int\limits_{\mathbb{R}^n}|(-\Delta)^{\sigma}f(x)|^2|x|^{2\sigma}\,dx \right).
\end{equation*}
Lastly, replacing \(f\) by \((-\Delta)^{-\sigma}f\), we obtain the desired result.

In the case of the Heisenberg group, the authors in \cite{Roncal-thangavelu-Hardy_inequality} established versions of Hardy's inequality for the Heisenberg group for non-homogeneous and homogeneous weights. First, the inequality was established for the conformally invariant fractional powers of the sublaplacian, denoted by \(\mathcal{L}_s\). The operator \(\mathcal{L}_s\) arises naturally in the context of CR geometry on the Heisenberg group. One can refer to \cite{Branson-inequalities_on_CR_Sphere, Frank_and_Lieb-sharp_constants_on_Hn} for more details.
We recall the definition of \(\mathcal{L}_s\) here. 

For a function \(f\) on \(\mathbb{H}^n\) and \(\lambda \in \mathbb{R}^*\), we denote \(\widehat{f}(\lambda)\) as its usual Fourier transform, which is operator valued, details of which can be found in \cite{thangaveluheisenberg}. The Fourier transform acts on \(\mathcal{L}f\) as
\begin{equation*}
     \widehat{(\mathcal{L}f)}(\lambda)=\widehat{f}(\lambda)H(\lambda).
\end{equation*}
\noindent Using the spectral decomposition of the Hermite operator \(H(\lambda)\), the spectral decomposition of the sublaplacian \(\mathcal{L}\) can be written as 
\begin{equation*}
\mathcal{L}f(z,t) = (2\pi)^{-n-1} \int\limits_{-\infty}^{\infty} \left( \sum_{k=0}^{\infty} (2k+n)|\lambda| f^{\lambda} * \varphi_{k}^{\lambda}(z) \right) e^{-i\lambda t} |\lambda|^n d\lambda.
\end{equation*}
The fractional power of the sublaplacian is defined via the spectral decomposition in a canonical way as 
\begin{equation*}
    \mathcal{L}^s f(z,w) = (2\pi)^{-n-1} \int\limits_{-\infty}^{\infty} \left( \sum_{k=0}^{\infty} \left((2k+n)|\lambda|\right)^s f^{\lambda} * \varphi_{k}^{\lambda}(z) \right) e^{-i\lambda w} |\lambda|^n d\lambda.
\end{equation*}

\noindent Note that \(\widehat{(\mathcal{L}^sf)}(\lambda)=\widehat{f}(\lambda)H^s(\lambda)\).

For \(0 \leq s <(n +1)\) the operator \(\mathcal{L}_s\) is defined as 
\begin{equation}\label{definition of L_s}
    \mathcal{L}_s f(z,w) = (2\pi)^{-n-1} \int\limits_{-\infty}^{\infty} \left( \sum_{k=0}^{\infty} (2|\lambda|)^s \frac{\Gamma\left(\frac{2k+n}{2} + \frac{1+s}{2}\right)}{\Gamma\left(\frac{2k+n}{2} + \frac{1-s}{2}\right)} f^{\lambda} * \varphi_{k}^{\lambda}(z) \right) e^{-i\lambda w} |\lambda|^n d\lambda.
\end{equation}
Taking the inverse Fourier transform in the central variable, we get 
\begin{equation}
\int\limits_{-\infty}^{\infty} \mathcal{L}_s f(z,w) e^{i\lambda w} dw = (2\pi)^{-n} |\lambda|^n \sum_{k=0}^{\infty} (2|\lambda|)^s \frac{\Gamma\left(\frac{2k+n}{2} + \frac{1+s}{2}\right)}{\Gamma\left(\frac{2k+n}{2} + \frac{1-s}{2}\right)} f^{\lambda} * \varphi_{k}^{\lambda}(z).
\end{equation}
The operator \(\mathcal{L}_s\) is known to admit explicit fundamental solution \cite{Cowling-comp_bdd_multipliers}.

 Motivated by the case of \(\mathbb{R}^n\), we deduce an analogue of inequality \eqref{pitts laplacian} for non-zero function \(f\in \mathcal{S}(\mathbb{H}^n)\) and \(p=q=2\). 
 
 We start with
$$\langle \mathcal{L}_{-s}f, f\rangle = \int\limits_{-\infty}^{\infty}(2\pi)^{-1}\langle L_{-s}^\lambda f^\lambda, f^\lambda\rangle \,d\lambda$$
where $L_{-s}^\lambda f^\lambda = (\mathcal{L}_{-s}f)^\lambda$ and is explicitly given by
\begin{equation}\label{expression for Lslambda}
    L_{-s}^\lambda f^\lambda(z) = (2|\lambda|)^{-s} \, (2\pi)^{-n} |\lambda|^n \sum_{k=0}^{\infty} \frac{\Gamma\left(\frac{2k+n+1-s}{2}\right)}{\Gamma\left(\frac{2k+n+1+s}{2}\right)} f^\lambda *_\lambda \varphi_{k,\lambda}^{n-1}(z).
\end{equation}
For \(0<s<1\), we use the following expansion of the Macdonald function \cite{thangavelu-roncal-oscar-Hardy_inequalities}
\begin{equation}
\pi^{-1/2} 2^{s+n-1} \frac{\Gamma(\frac{n-1-s}{2})}{\Gamma(s)} (|\lambda| |z|^2)^{-(n-s)/2} K_{(n-s)/2} \left( \frac{|\lambda|}{2} |z|^2 \right) 
= \frac{2}{\Gamma(n)} \sum_{k=0}^{\infty} \frac{\Gamma\left(\frac{2k+n+1-s}{2}\right)}{\Gamma\left(\frac{2k+n+1+s}{2}\right)}\varphi_{k,\lambda}^{n-1}(z).
\end{equation}
Using this, we see that
$$L_{-s}^\lambda f^\lambda(z) = (2|\lambda|)^{-s} \, (2\pi)^{-n} |\lambda|^n \, f^\lambda *_\lambda k_s(z)$$
with the kernel $k_s$ given by
$$k_s(z) = \pi^{-1/2} 2^{s+n-1} \frac{\Gamma(n) \Gamma(\frac{(n-1-s)}{2})}{\Gamma(s)} (|\lambda| |z|^2)^{-(n-s)/2} K_{(n-s)/2} \left( \frac{|\lambda|}{2} |z|^2 \right)$$
From \cite{handbook}, we have that for \(y>0\) the Macdonald function $K_\nu$ has the integral representation
$$ K_\nu(y)=\frac{\Gamma(\nu + 1/2)}{\Gamma(1/2)}(2y)^\nu \int_{0}^{\infty} \frac{\cos t}{(t^2 + y^2)^{\nu+1/2}} \, dt. $$
Using the bound for the integral on the right, we obtain the estimate 
$$ K_\nu(y) \leq \frac{\Gamma(\nu + 1/2)}{\Gamma(1/2)}(2y)^\nu y^{-2\nu} \int_{0}^{\infty} (t^2 + 1)^{-\nu-1/2} \, dt. $$
Now,
$$ 2 \int_{0}^{\infty} (t^2 + 1)^{-\nu-1/2} \, dt = \int_{0}^{\infty} \frac{t^{-1/2}}{(1 + t)^{\nu+1/2}} \, dt = B(1/2, \nu)=\frac{\Gamma(1/2)\Gamma(\nu)}{\Gamma(\nu+1/2)} .$$
This gives the estimate $K_\nu(y) \leq \Gamma(\nu) 2^{\nu-1} y^{-\nu}$. 

\noindent Subsequently, from this we deduce  $k_s(z) \leq d_{n,s} (|\lambda|)^{-n+s} |z|^{-2n+2s}$ where
$$ d_{n,s} = 2^{2n-2} \frac{\Gamma(\frac{n-1-s}{2})\Gamma(n) \Gamma(n-s)}{\Gamma(s)\Gamma(1/2)}. $$
Therefore, we have 
\begin{align*}
    \left| \langle L_{-s}^\lambda f^\lambda, f^\lambda \rangle \right|&=|\lambda|^{n-s}\int_{\mathbb{C}^n}|f^\lambda| \ast k_s(z) |f^\lambda|(z)dz\\
    & \leq d_{n,s} \int_{\mathbb{C}^n}(-\Delta^{-s})|f^\lambda|(z) |f^\lambda|(z)dz\\
    &=d_{n,s}\langle (\Delta^{-s})|f^\lambda|,|f^\lambda| \rangle\\
    &\leq d_{n,s}\pi^{2s}\left(\frac{\Gamma(\frac{n-2s}{4})}{\Gamma(\frac{n+2s}{4})}\right)^2\int_{\mathbb{C}^n}|z|^{2s}|f^\lambda(z)|^2dz.
\end{align*}
The last inequality follows from the Pitt's inequality on \(\mathbb{R}^{2n}\). 

\noindent Thus we have 
\begin{equation}\label{pitts with weight |z|-2s}
    \langle \mathcal{L}_{-s}f,f\rangle \leq d_{n,s}\pi^{2s}\left(\frac{\Gamma(\frac{n-2s}{4})}{\Gamma(\frac{n+2s}{4})}\right)^2 \int_{\mathbb{H}^n}|z|^{2s}|f(z,t)|^2\,dz\,dt. 
\end{equation}
The above inequality is an analogue of the inequality \eqref{pitts laplacian}.
Also, it leads to the following Hardy's inequality
\begin{equation}\label{equiv hardy for weight |z|2s}
    \langle \mathcal{L}_sf,f\rangle \geq d_{n,s}^{-1} \pi^{-2s} \left(\frac{\Gamma(\frac{n+2s}{4})}{\Gamma(\frac{n-2s}{4})}\right)^2 \int_{\mathbb{H}^n}|z|^{-2s}|f(z,t)|^2\,dz\,dt,
\end{equation}
which is analogue of \eqref{hardy ineq}.

We can also reformulate \eqref{pitts with weight |z|-2s} in terms of the Strichartz Fourier transform on \(\mathbb{H}^n\).
Using the expression for \(L^{\lambda}_{-s}f^\lambda\) as in \eqref{expression for Lslambda}, we obtain the identity
\begin{equation}\label{expression L_-sf,f}
    \langle \mathcal{L}_{-s}f,f\rangle=(2\pi)^{-1}\int\limits_{-\infty}^{\infty}(2|\lambda|)^{-s}(2\pi)^{-2n}|\lambda|^{2n} \sum\limits_{k=0}^\infty \frac{\Gamma\left(\frac{2k+n+1-s}{2}\right)}{\Gamma\left(\frac{2k+n+1+s}{2}\right)} \lVert f^{\lambda}\ast_{\lambda}\varphi^{n-1}_{k,\lambda}\rVert_2^2\,d\lambda.
\end{equation}
Moreover we have 
$$
\lVert f^\lambda \ast_{\lambda} \varphi^{n-1}_{k,\lambda} \rVert_2^2=(2\pi)^{n}|\lambda|^{-n}\lVert \widehat{f}(\lambda)P_k(\lambda)\rVert_{HS}^2,
$$
where \(P_k(\lambda)\) denotes the orthogonal projection of \(L^2(\mathbb{R}^n)\) on the \(k\)-th eigenspace of the Hermite operator \(H(\lambda)\).

Using this in \eqref{expression L_-sf,f}, we obtain
\begin{align}
\langle \mathcal{L}_{-s}f,f\rangle&=(2\pi)^{-1}\int\limits_{-\infty}^{\infty}(2|\lambda|)^{-s}(2\pi)^{-n}|\lambda|^{n} \sum\limits_{k=0}^\infty \frac{\Gamma\left(\frac{2k+n+1-s}{2}\right)}{\Gamma\left(\frac{2k+n+1+s}{2}\right)} \lVert \widehat{f}(\lambda)P_k(\lambda)\rVert_{HS}^2\,d\lambda\nonumber\\
&=(2\pi)^{-1}\int\limits_{-\infty}^{\infty}(2|\lambda|)^{-s}(2\pi)^{-2n}|\lambda|^{2n} \sum\limits_{k=0}^\infty \frac{\Gamma\left(\frac{2k+n+1-s}{2}\right)}{\Gamma\left(\frac{2k+n+1+s}{2}\right)}\left(c_{n,k}^{-2}\int\limits_{\mathbb{C}^n}\lvert\tilde{f}(\lambda,(2k+n)|\lambda|),w)\rvert^2 \,dw\right)\,d\lambda\nonumber \\
&=\int\limits_{\Omega \times \mathbb{C}^n}(2|\lambda|)^{-s}\left(\frac{\Gamma\left(\frac{2k+n+1-s}{2}\right)}{\Gamma\left(\frac{2k+n+1+s}{2}\right)}\right)\left| \tilde{f}(a,w)\right|^2 d\nu_2(a)\;dw .
\end{align}
The last equation uses the following relation from \cite{Strichartz_Fourier_thangavelu}
\begin{equation}\label{strichartz relation formula}
\int\limits_{\mathbb{C}^n}\left|\tilde{f}(a,w) \right|^2dw=c_{n,k}^2\int\limits_{\mathbb{C}^n}\left|\widehat{f}(a,w) \right|^2dw=c_{n,k}^2(2\pi)^n|\la|^{-n}\left \lVert \widehat{f}(\lambda)P_k(\lambda)\right \rVert^2_{HS}.
\end{equation}

We summarize this result in the following theorem.
\begin{theorem}\label{theorem pitts weight |z|2s}
    Let \(f \in \mathcal{S}(\mathbb{H}^n)\) and \(0<s<1\). Then the following Pitt's inequality holds 
    \begin{align*}
   \int\limits_{\Omega \times \mathbb{C}^n}(2|\lambda|)^{-s}&\left(\frac{\Gamma\left(\frac{2k+n+1-s}{2}\right)}{\Gamma\left(\frac{2k+n+1+s}{2}\right)}\right)\left| \tilde{f}(a,w)\right|^2 d\nu_2(a)\;dw \\
&\leq  d_{n,s}\pi^{2s}\left(\frac{\Gamma(\frac{n-2s}{4})}{\Gamma(\frac{n+2s}{4})}\right)^2 \int_{\mathbb{H}^n}|z|^{2s}|f(z,t)|^2\,dz\,dt. 
        \end{align*}
\end{theorem}
\noindent We can derive another version using the Pitt's inequality on the real line. 
Recall from \eqref{expression L_-sf,f}, we have the identity
\begin{equation*}
    \langle \mathcal{L}_{-s}f,f\rangle=(2\pi)^{-1}\int\limits_{-\infty}^{\infty}(2|\lambda|)^{-s}(2\pi)^{-2n}|\lambda|^{2n} \sum\limits_{k=0}^\infty \frac{\Gamma\left(\frac{2k+n+1-s}{2}\right)}{\Gamma\left(\frac{2k+n+1+s}{2}\right)} \lVert f^{\lambda}\ast_{\lambda}\varphi^{n-1}_{k,\lambda}\rVert_2^2\,d\lambda.
\end{equation*}
Consequently using the special Hermite expansion, we obtain 
\begin{equation*}
     \langle \mathcal{L}_{-s}f,f\rangle \leq (2\pi)^{-1}\int\limits_{\mathbb{C}^n}\int\limits_{-\infty}^{\infty}(2|\lambda|^{-s})|f^{\lambda}(z)|^2\,d\lambda\,dz.
\end{equation*}
Applying the Pitt's inequality for the Fourier transform on the real line, we obtain
\begin{equation}\label{pitts for weight |(z,t)|2s}
    \langle \mathcal{L}_{-s}f,f\rangle \leq 2^{-s} \pi^{s}\left( \frac{\Gamma(\frac{1-s}{4})}{\Gamma(\frac{1+s}{4})}\right)^2\int\limits_{\mathbb{C}^n}\int\limits_{-\infty}^{\infty}|t|^{s}|f(z,t)|^2\,dz\,dt \leq 2^{-s} \pi^{s}\left( \frac{\Gamma(\frac{1-s}{4})}{\Gamma(\frac{1+s}{4})}\right)^2 \int\limits_{\mathbb{H}^n} |(z,t)|^{2s} |f(z,t)|^2\,dz\,dt.
\end{equation}
The above inequality is analogue of Pitt's inequality \eqref{pitts laplacian} for the weight \(|(z,t)|^{-2s}\). Moreover it leads to the Hardy's inequality 
\begin{equation}\label{equiv hardy for weight |(z,t)|2s}
    \langle \mathcal{L}_sf,f \rangle \geq 2^s \pi^{-s}\left( \frac{\Gamma(\frac{1+s}{4})}{\Gamma(\frac{1-s}{4})}\right)^2\int\limits_{\mathbb{H}^n} |(z,t)|^{-2s} |f(z,t)|^2\,dzdt.
\end{equation}
Following the previous method we obtain the following theorem for the weight \(|(z,t)|^{-2s}\).
 \begin{theorem}\label{theorem pitts weight |(z,t)|2s}
     Let \(f\in \mathcal{S}(\mathbb{H}^n)\) and \(0<s< n+1\). Then the following Pitt's inequality holds
      \begin{align*}
    \int\limits_{\Omega \times \mathbb{C}^n}(2|\lambda|)^{-s}&\left(\frac{\Gamma\left(\frac{2k+n+1-s}{2}\right)}{\Gamma\left(\frac{2k+n+1+s}{2}\right)}\right)\left| \tilde{f}(a,w)\right|^2 d\nu_2(a)\;dw \\
&\leq  2^{-s} \pi^{s}\left( \frac{\Gamma(\frac{1-s}{4})}{\Gamma(\frac{1+s}{4})}\right)^2 \int\limits_{\mathbb{H}^n} |(z,t)|^{2s} |f(z,t)|^2\,dz\,dt. 
        \end{align*}   
\end{theorem}
\begin{remark}
    The following remarks regarding the preceding two theorems are in order.
    \begin{enumerate}
        \item An \(L^p\)- version of \eqref{pitts with weight |z|-2s} with sharp constants has been obtained in \cite{Roncal-Chunxia-Sharp_Lp_Hardy_Stein_Weiss}.
        \item The approach used here to derive Pitt’s inequalities in Theorems \ref{theorem pitts weight |z|2s} and \ref{theorem pitts weight |(z,t)|2s} is based on H\"older’s inequality combined with suitable constants. While this method yields valid bounds, it does not provide sharpness of the constants.
    \end{enumerate}
\end{remark}
There is yet another way to derive Hardy's inequality for \(\mathcal{L}_s\) with a different constant other than the ones obtained from the preceding two methods.
We follow the methods and notations of \cite{Roncal-Thangavelu-Trace_hardy_IMRN}.
For \(\gamma>0\) and \(\rho>0\), define 
\begin{equation*}
    \psi_{s,\gamma}(z,t)=(\gamma+|z|^4+\gamma t^2)^{-\frac{n+1+s}{4}} \text{ \, and \,} \varphi_{_s,\rho}(z,t)=\left((\rho^2+|z|^2)^2+16t^2 \right)^{-\frac{n+1+s}{2}}.
\end{equation*}
Following \cite{Roncal-Thangavelu-Trace_hardy_IMRN}, if we choose  
\begin{equation*}
    c(n,s)=4 \pi^{-n-1/2}\frac{\Gamma(n+s)\Gamma(\frac{n+1+s}{2})}{\Gamma(s)\Gamma(\frac{n+s}{2})},
\end{equation*}
then integral of \(c(n,s)\rho^{2s}\varphi_{s,\rho}\) equals 1. Moreover, \(c(n,s)\rho^{2s}\varphi_{s,\rho}\) forms an approximate identity.

\noindent By replacing \(\varphi_s\) by \(\psi_{s,\gamma}\) in proof of \cite[Theorem 1.5]{Roncal-Thangavelu-Trace_hardy_IMRN} and adapting the same techniques as in \cite[Corollary 1.6]{Roncal-Thangavelu-Trace_hardy_IMRN}, we obtain the following proposition.
\begin{proposition}
    Let \(0<s<1\) and \(f,\mathcal{L}_sf \in L^2(\mathbb{H}^n)\) then
    \begin{equation}\label{analogue cor1.6}
    \langle \mathcal{L}_sf,f\rangle \geq  c(n,s)^{-1} 4^{2s} \left( \frac{\Gamma(\frac{n+1+s}{2})}{\Gamma(\frac{n+1-s}{2})}\right)^2 \int\limits_{\mathbb{H}^n}\frac{\psi_{s,\gamma}(z,t)}{\psi_{s,\gamma}\ast \lvert \cdot \rvert^{-Q+2s}(z,t)}\lvert f(z,t)\rvert^2\,dz\,dt.
\end{equation}
\end{proposition}
In the above equation, \(Q=2n+2\) refers to the homogeneous dimension of \(\mathbb{H}^n\). We have not included the proof of this proposition as it is a verbatim repetition of \cite[Corollary 1.6]{Roncal-Thangavelu-Trace_hardy_IMRN} only with some obvious changes.

Using Fatou's lemma in \eqref{analogue cor1.6}, we deduce the inequality
\begin{equation}\label{L_s inequality after fatou's lemma}
   \langle \mathcal{L}_sf,f\rangle \geq  c(n,s)^{-1} 4^{2s} \left( \frac{\Gamma(\frac{n+1+s}{2})}{\Gamma(\frac{n+1-s}{2})}\right)^2 \int\limits_{\mathbb{H}^n}\frac{\psi_{s,0}(z,t)}{\psi_{s,0}\ast \lvert \cdot \rvert^{-Q+2s}(z,t)}\lvert f(z,t)\rvert^2\,dz\,dt. 
\end{equation}
Note that \(\psi_{s,0}(z,t)=|z|^{-(n+1+s)}\) a choice that allows us to calculate \(|\cdot |^{-Q+2s} \ast \psi_{s,0}(z,t)=\psi_{s,0} \ast | \cdot |^{-Q+2s}(z,t)\).

We have 
\begin{align*}
    |\cdot |^{-Q+2s} \ast \psi_{s,0}(z,t)&=\int\limits_{\mathbb{C}^n}\int\limits_{-\infty}^{\infty}\lvert (z,t)(w,b)^{-1}\rvert^{-Q+2s} |w|^{-(n+1+s)}\,db\,dw\\
    &=\int\limits_{\mathbb{C}^n}\int\limits_{-\infty}^{\infty}\lvert (z,t)(-w,-b)\rvert^{-Q+2s} |w|^{-(n+1+s)}\,db\,dw.
\end{align*}
We can calculate the inner integral explicitly as
\begin{align*}
    \int\limits_{-\infty}^{\infty}\lvert (z,t)(-w,-b)\rvert\,db&=\int\limits_{\mathbb{C}^n} \left( |z-w|^4+16(t-b-\operatorname{Im}(z\bar{w}))^2\right)^{-\frac{(n+1-s)}{2}}\,db\\
    &=\int\limits_{\mathbb{C}^n} \left( |z-w|^4+16b^2\right)^{-\frac{(n+1-s)}{2}}\,db.
\end{align*}
 Using the fact that
 $$ \int\limits_{-\infty}^\infty (1+b^2)^{-\nu-1/2} db = B(1/2, \nu), $$ 
 where \(\nu>0\) and \(B(1/2,\nu)\) refers to the beta function evaluated at \((1/2,\nu)\).
 We see that
$$ \int\limits_{-\infty}^\infty |(z,t)(-w,-b)|^{-Q+2s}\,db =  \frac{1}{4}\, B(1/2, (n-s)/2) |z-w|^{-2n+2s}.$$
Therefore,
\begin{equation*}
     |\cdot|^{-Q+2s} \ast \psi_{s,0}(z) =\frac{1}{4}\, B(1/2, (n-s)/2)  \int\limits_{\mathbb{C}^n} |z-w|^{-2n+2s}\, |w|^{-2n+(n-1-s)}\, dw.
\end{equation*}
Next, we make use of the following formula (see \cite[Chapter 5]{book-fractional_integrals})
\begin{equation*}
    \int\limits_{\mathbb{C}^n} |z-w|^{-2n+\alpha}\, |w|^{-2n+\beta}\, dw = \frac{ \zeta(\alpha)\zeta(\beta)}{\zeta(\alpha+\beta)} \, |z|^{-2n+(\alpha+\beta)},
\end{equation*}
where
$$ \zeta(\alpha) =  \frac{\pi^n 2^\alpha \Gamma(\alpha/2)}{\Gamma((2n-\alpha)/2)}.$$
This implies 
$$  |\cdot|^{-Q+2s} \ast \psi_{s,0}(z) \,  =  \frac{1}{4}\, B(1/2, (n-s)/2)  \frac{ \zeta(n-1-s) \zeta(2s)}{\zeta(n-1+s)} \, |z|^{-(n+1-s)}.$$
Consequently applying this in \eqref{L_s inequality after fatou's lemma}, we obtain
$$
\langle \mathcal{L}_{s}f, f \rangle \geq  c(n,s)^{-1}\, b(n,s)^{-1} \,4^{2s} \Big( \frac{\Gamma(\frac{n+1+s}{2})}{\Gamma(\frac{n+1-s}{2})} \Big)^2 \, \int\limits_{\mathbb{H}^n}  |z|^{-2s} \,  |f(z,t)|^2 \, dz \,dt,
$$
where we have written
$$ b(n,s) = \frac{1}{4}\, B(1/2, (n-s)/2)  \frac{ \zeta(n-1-s) \zeta(2s)}{\zeta(n-1+s)}.  $$ 
A simple  calculation  shows that 
$$ b(n,s) = \frac{\pi^{n+1/2}}{4} \, \frac{ \Gamma(\frac{n-s}{2})}{\Gamma(n-s)} \,\frac{ \Gamma(\frac{n-1-s}{2})}{\Gamma( \frac{n-1+s}{2})}\, \frac{ \Gamma(s)}{\Gamma( \frac{n+1+s}{2})}.$$
Moreover, using Legendre's duplication formula 
$$ \Gamma(z) \Gamma(z+1/2) = 2^{1-2z} \sqrt{\pi} \Gamma(2z) $$ 
and simplifying we get
$$ b(n,s) c(n,s) = 4^s    \,\frac{ \Gamma(\frac{n-1-s}{2})}{\Gamma( \frac{n-1+s}{2})}\,\,\frac{ \Gamma(\frac{n+1+s}{2})}{\Gamma( \frac{n+1-s}{2})}\,=  4^s  \,\frac{n-1+s}{n-1-s}.$$
Finally we get the following Hardy inequality: 
 \begin{equation}\label{hardy via trace hardy} 
 \langle\mathcal{L}_{s}f, f \rangle \geq   \,4^{s} \, \frac{n-1-s}{n-1+s} \left(\frac{\Gamma(\frac{n+1+s}{2})}{\Gamma(\frac{n+1-s}{2})}\right)^2  \,  \,  \int\limits_{\mathbb{H}^n}  |z|^{-2s} \,  |f(z,t)|^2 \, dz\, dt.
  \end{equation}

Using the above, we can write the Pitt's inequality 
\begin{equation}\label{pitts via trace hardy}
\langle \mathcal{L}_{-s}f,f \rangle\leq 4^{-s} \, \frac{n-1+s}{n-1-s} \left(\frac{\Gamma(\frac{n+1-s}{2})}{\Gamma(\frac{n+1+s}{2})}\right)^2  \,  \,  \int\limits_{\mathbb{H}^n}  |z|^{2s} \,  |f(z,t)|^2 \, dz\, dt.
\end{equation}
Writing in terms of the Strichartz Fourier transform, we obtain the following theorem.
\begin{theorem}
    Let \(f\in \mathcal{S}(\mathbb{H}^n)\) and \(0<s<1\) then the following Pitt's inequality hold
    \begin{align*}
   \int\limits_{\Omega \times \mathbb{C}^n}(2|\lambda|)^{-s}&\left(\frac{\Gamma\left(\frac{2k+n+1-s}{2}\right)}{\Gamma\left(\frac{2k+n+1+s}{2}\right)}\right)\left| \tilde{f}(a,w)\right|^2 d\nu_2(a)\;dw \\
&\leq  (2\pi)^{-1}4^{-s} \, \frac{n-1+s}{n-1-s} \left(\frac{\Gamma(\frac{n+1-s}{2})}{\Gamma(\frac{n+1+s}{2})}\right)^2  \,  \,  \int\limits_{\mathbb{H}^n}  |z|^{2s} \,  |f(z,t)|^2 \, dz\, dt. 
        \end{align*}
\end{theorem}
 The constant obtained in the inequality \eqref{hardy via trace hardy} may or may not be sharp, so we obtain bounds for the sharp constant in the \eqref{hardy via trace hardy}. Suppose we have
  $$ (\mathcal{L}_{s}f, f ) \geq  C\,  \,  \int\limits_{\mathbb{H}^n}  |z|^{-2s} \,  |f(z,t)|^2 \, dz \,dt,$$
  for sharp constant \(C\).
  
  \noindent By taking $ f = \varphi_{-s,\rho} $  in the above inequality, we get
  $$ \rho^{2s}\,C_2(n,s) \int\limits_{\mathbb{H}^n}  \left((\rho^2+|z|^2)^2+16t^2\right)^{-(n+1)} \,dz\, dt \geq C \int\limits_{\mathbb{H}^n} \varphi_{-s,\rho}(z,t)^2 \, |z|^{-2s} \, dz\, dt, $$
  where 
  $$ C_2(n,s) = 4^{2s} \Big( \frac{\Gamma(\frac{n+1+s}{2})}{\Gamma(\frac{n+1-s}{2})} \Big)^2 .$$
  Note that we have used the Cowling-Haagerup identity \cite[Theorem 3.7]{Roncal-Thangavelu-Trace_hardy_IMRN} to simplify the expression for \(\langle \mathcal{L}_sf,f \rangle\).
  
  \noindent By applying change of variables, the above leads to the inequality
  $$ C \leq C_2(n,s) \frac { \displaystyle\int\limits_{\mathbb{H}^n}  \left((1+|z|^2)^2+16t^2\right)^{-(n+1)} \,dz \,dt }{ \displaystyle\int\limits_{\mathbb{H}^n}  \left((1+|z|^2)^2+16t^2\right)^{-(n+1-s)} \, |z|^{-2s}\,dz\, dt }.$$
Furthermore change of variables $ t \rightarrow \frac{1}{4}(1+|z|^2) t $ simplifies the ratio on  the right hand side into
  $$  \frac { \displaystyle\int\limits_{-\infty}^\infty  \left(1+t^2\right)^{-(n+1)}  \,dt }{\displaystyle \int\limits_{-\infty}^\infty  \left(1+t^2\right)^{-(n+1-s)}  \,dt } \,\,\,\frac { \displaystyle\int\limits_{\mathbb{C}^n}  (1+|z|^2)^{-2n-1}\, dz }{ \displaystyle\int\limits_{\mathbb{C}^n}  (1+|z|^2)^{-2n+2s-1} \, |z|^{-2s}\,dz } . $$
  We compute the above integrals explicitly.
  Recall the formula \cite[(5.12.3)]{frank-oliver-handbook}
  $$
  \int\limits_0^\infty (1+t)^{-b}t^{a-1}\,dt=\frac{\Gamma(a)\Gamma(b-a)}{\Gamma(b)},
  $$
  for \(a>0,\;b>a\).
 Using this, we obtain
  $$
    \frac { \displaystyle\int\limits_{-\infty}^\infty  \left(1+t^2\right)^{-(n+1)} \, dt }{ \displaystyle\int\limits_{-\infty}^\infty  \left(1+t^2\right)^{-(n+1-s)}\,  dt }= \frac{\Gamma(1/2) \Gamma(n+1/2)}{\Gamma(n+1)} \frac{ \Gamma(n-s+1)}{ \Gamma(1/2)\Gamma(n-s+1/2)} .
  $$
  Moreover, writing the remaining integrals in terms of polar coordinates, we have
$$\frac { \displaystyle\int\limits_{\mathbb{C}^n}  (1+|z|^2)^{-2n-1} \,dz }{ \displaystyle\int\limits_{\mathbb{C}^n}  (1+|z|^2)^{-2n+2s-1} \, |z|^{-2s}\,dz }=\frac { \displaystyle\int\limits_{0}^\infty  \left(1+r^2\right)^{-2n-1}  r^{2n-1} dr}{ \displaystyle\int\limits_{0}^\infty  \left(1+r^2\right)^{-2n+2s-1}  r^{2n-2s-1}dr }.$$
After evaluation we get  
 $$\frac { \displaystyle\int\limits_{\mathbb{C}^n}  (1+|z|^2)^{-2n-1} \,dz }{ \displaystyle\int\limits_{\mathbb{C}^n}  (1+|z|^2)^{-2n+2s-1} \, |z|^{-2s}\,dz } =
  \frac{\Gamma(n) \Gamma(n+1)}{\Gamma(2n+1)} \frac{ \Gamma(2n-2s+1)}{ \Gamma(n-s)\Gamma(n+1-s)} .$$
  Therefore, we deduce the inequality 
  $$ C \leq C_2(n,s)\,   \frac{ \Gamma(n+1/2)}{\Gamma(2n+1)}   \frac{ \Gamma(2n-2s+1)}{\Gamma(n-s+1/2)}  \frac{ \Gamma(n)}{ \Gamma(n-s)} . $$
  In view of Legendre's duplication formula, the above expression can be simplified as 
  $$ C \leq C_2(n,s)\, \frac{\sqrt{\pi} 2^{-2n}}{\Gamma(n+1)}\, \frac{2^{2n-2s} \Gamma(n+1-s)}{\sqrt{\pi}}\, \frac{ \Gamma(n)}{ \Gamma(n-s)} = C_2(n,s) 4^{-s} \frac{n-s}{n} .$$

\noindent Note that for $ 0 < s <  n+1 $, $ \frac{n-1-s}{n-1+s} < \frac{n-s}{n} $. Thus, for \(0<s<1\) we see that the sharp constant \(C\), satisfies the bounds
\begin{equation}\label{bounds for sharp constant}
\,4^{s} \, \frac{n-1-s}{n-1+s} \left(\frac{\Gamma(\frac{n+1+s}{2})}{\Gamma(\frac{n+1-s}{2})}\right)^2  \leq  C \leq  \,4^{s} \, \frac{n-s}{n} \left(\frac{\Gamma(\frac{n+1+s}{2})}{\Gamma(\frac{n+1-s}{2})}\right)^2.
\end{equation}

In continuation of the study of bounds for the sharp constant in Hardy's inequality, we would like to add a remark for the sharp constant in the Hardy's inequality with the weight $ |(z,a)|^{-2s} $.
\begin{remark}
If $ C_H $ denotes the sharp constant in the Hardy's inequality with the weight $ |(z,a)|^{-2s} $ we should have $ C \leq C_H. $ From \cite{corr-Roncal-Thangavelu-trace_hardy_imrn} we know $ C_H \leq C_2(n,s) $. Thus we have
 $$ 4^{s} \, \frac{n-1-s}{n-1+s} \left(\frac{\Gamma(\frac{n+1+s}{2})}{\Gamma(\frac{n+1-s}{2})}\right)^2  \leq  C_H \leq  \,4^{2s} \, \left(\frac{\Gamma(\frac{n+1+s}{2})}{\Gamma(\frac{n+1-s}{2})}\right)^2  .$$
\end{remark}

In further addition to our study of relation of Hardy's inequality and Pitt's inequality, we now state the Hardy inequalities for non-homogeneous weight, proved in \cite{Roncal-thangavelu-Hardy_inequality}. Denote \(W^{s,2}(\mathbb{H}^n)\) as the Sobolev space consisting of all square integrable functions \(f\) on \(\mathbb{H}^n\) for which \(\mathcal{L}^{s/2}f \in L^2(\mathbb{H}^n)\).
\begin{theorem}(For non-homogeneous weight)\label{hardy nonhomogeneous for L_s}
 Let \(0 < s < \frac{n+1}{2}\) and \(\delta > 0\). Then for all \(f \in W^{s,2}(\mathbb{H}^n)\) the following holds
\[
(4\delta)^s \frac{\Gamma\left(\frac{1+n+s}{2}\right)^2}{\Gamma\left(\frac{1+n-s}{2}\right)^2} \int\limits_{\mathbb{H}^n} \frac{|f(z,t)|^2}{\left(\left(\delta + \frac{1}{4}|z|^2\right)^2 + t^2\right)^s} dz\,dt \leq \langle \mathcal{L}_s f, f \rangle.
\]
\end{theorem}
Next using the operator \(U_s=\mathcal{L}_s \mathcal{L}^{-s}\), with
$$
\|U_s\| = \sup_{k \ge 0} \left(\frac{2k+n}{2}\right)^{-s} \frac{\Gamma\left(\frac{2k+n}{2} + \frac{1+s}{2}\right)}{\Gamma\left(\frac{2k+n}{2} + \frac{1-s}{2}\right)}, 
$$
the authors in \cite{Roncal-thangavelu-Hardy_inequality} derived the following analogue for \(\mathcal{L}^s\).

\begin{theorem}\label{Hardy nonhomogeneous for L^s}
Let \(0 < s < \frac{n+1}{2}\) and \(\delta>0\). Then for all \(f \in W^{s,2}(\mathbb{H}^n)\) the following holds
\[ (4\delta)^s \frac{\Gamma\left(\frac{1+n+s}{2}\right)^2}{\Gamma\left(\frac{1+n-s}{2}\right)^2} \int\limits_{\mathbb{H}^n} \frac{|f(z, t)|^2}{\left(\left(\delta + \frac{1}{4}|z|^2\right)^2 + t^2\right)^s} dz\,dt \le \|U_s\|\langle\mathcal{L}^s f, f\rangle.
\]
\end{theorem}

Following similar line of argument as before, we deduce an analogue of Pitts' inequality for non-zero function \(f\in \mathcal{S}(\mathbb{H}^n)\) and \(p=q=2\). 

\noindent Let 
$$
v(z,t)=\left(\left(\delta + \frac{1}{4}|z|^2\right)^2 + t^2\right)^{s/2}.
$$
From Theorem \ref{hardy nonhomogeneous for L_s}, we have 
$$
(4\delta)^s \frac{\Gamma\left(\frac{1+n+s}{2}\right)^2}{\Gamma\left(\frac{1+n-s}{2}\right)^2} \int\limits_{\mathbb{H}^n}\left (\frac{f(z,t)}{v(z,t)} \right)^2  dz\,dt \leq \langle \mathcal{L}_s f, f \rangle.
$$
Then corresponding Pitt's inequality follows as 
\begin{equation}\label{pitts non-homogeneous 1}
    \langle \mathcal{L}_{-s}f,f \rangle \leq (4\delta)^{-s} \frac{\Gamma\left(\frac{1+n-s}{2}\right)^2}{\Gamma\left(\frac{1+n+s}{2}\right)^2} \int\limits_{\mathbb{H}^n}\left (\frac{f(z,t)}{v(z,t)} \right)^2  dz\,dt.
\end{equation}
In terms of Strichartz Fourier transform, we obtain
\begin{equation*}
    \int\limits_{\Omega \times \mathbb{C}^n}(2|\lambda|)^{-s}\left(\frac{\Gamma\left(\frac{2k+n+1-s}{2}\right)}{\Gamma\left(\frac{2k+n+1+s}{2}\right)}\right)\left| \tilde{f}(a,w)\right|^2 d\nu_2(a)\;dw \leq (4\delta)^{-s} \frac{\Gamma\left(\frac{1+n-s}{2}\right)^2}{\Gamma\left(\frac{1+n+s}{2}\right)^2} \int\limits_{\mathbb{H}^n}\left (\frac{f(z,t)}{v(z,t)} \right)^2  dz\,dt.
\end{equation*}

\noindent Using Theorem \ref{Hardy nonhomogeneous for L^s} and a similar technique, we obtain another analogue of Pitt's inequality with non-homogeneous weight as follows 
\begin{equation}\label{Pitts non-homogeneous 2}
 \int\limits_{\Omega \times \mathbb{C}^n}\left( (2k+n)|\lambda|\right)^{-s}   \left|\tilde{f}(a,w) \right|^2 d\nu_2(a) \;dw\leq \lVert U_s \rVert \left( (4\delta)^s \frac{\Gamma\left(\frac{1+n+s}{2}\right)^2}{\Gamma\left(\frac{1+n-s}{2}\right)^2}\right)^{-1} \int\limits_{\mathbb{H}^n}\left|v(z,t)f(z,t)\right|^2dz\;dt.
\end{equation}
Besides the non-homogeneous weight, Hardy's inequality for homogeneous weight has also been proved for \(\mathcal{L}^s\) in the same paper \cite{Roncal-thangavelu-Hardy_inequality}.
\begin{theorem}(Hardy inequality for homogeneous weight)
 Let \(0 < s < 1\). Then for all $f \in C_{0}^{\infty}(\mathbb{H}^n)$,
$$
\frac{2^{2n+3s}\Gamma(\frac{n+s}{2})^2}{\Gamma(1-s)\Gamma(\frac{n}{2})^2} \int_{\mathbb{H}^n} \frac{|f(z,t)|^2}{|(z,t)|^{2s}} dz\,dt \leq \|V_s\| \langle \mathcal{L}^s f, f \rangle.
$$  
\end{theorem}
Using the same line of reasoning as in the case of the non-homogeneous weight case, one can derive the form of Pitt's inequality with the homogeneous weight \(|(z,w)|^s\), for \(0<s<1\). Namely, 
\begin{equation}\label{Pitt's homogeneous}
 \int\limits_{\Omega \times \mathbb{C}^n}\left( (2k+n)|\lambda|\right)^{-s}   \left|\tilde{f}(a,w) \right|^2 d\nu_2(a) \;dw\leq (2\pi)^{-n}\lVert V_s \rVert  \left( \frac{2^{2n+3s}\Gamma(\frac{n+s}{2})^2}{\Gamma(1-s)\Gamma(\frac{n}{2})^2}\right)^{-1}  \int\limits_{\mathbb{H}^n}\left|(z,t) \right|^{2s}\left(f(z,t)\right)^2dz\;dt.
\end{equation}

 \noindent We summarize the results in the inequalities \eqref{pitts non-homogeneous 1}, \eqref{Pitts non-homogeneous 2} and \eqref{Pitt's homogeneous} in the following theorem.
\begin{theorem}
Let \(f \in \mathcal{S}(\mathbb{H}^n)\). Then the following hold
\begin{itemize}
    \item non-homogeneous case: Let \(0 < s < \frac{n+1}{2}\) and \(\delta > 0\). Then the following hold 
    \begin{enumerate}
        \item[(a)] \begin{align*}
\int\limits_{\Omega \times \mathbb{C}^n}(2|\lambda|)^{-s}\left(\frac{\Gamma\left(\frac{2k+n+1-s}{2}\right)}{\Gamma\left(\frac{2k+n+1+s}{2}\right)}\right)&\left| \tilde{f}(a,w)\right|^2 d\nu_2(a)\;dw\\
&\leq (4\delta)^{-s} \frac{\Gamma\left(\frac{1+n-s}{2}\right)^2}{\Gamma\left(\frac{1+n+s}{2}\right)^2} \int\limits_{\mathbb{H}^n}\left (\frac{f(z,t)}{v(z,t)} \right)^2  dz\,dt. 
        \end{align*}
        \item[(b)] \begin{align*}
 \int\limits_{\Omega \times \mathbb{C}^n}&\left( (2k+n)|\lambda|\right)^{-s}   \left|\tilde{f}(a,w) \right|^2 d\nu_2(a) \;dw\\
& \leq \lVert U_s \rVert \left( (4\delta)^s \frac{\Gamma\left(\frac{1+n+s}{2}\right)^2}{\Gamma\left(\frac{1+n-s}{2}\right)^2}\right)^{-1}\int\limits_{\mathbb{H}^n}\left(\left(\delta + \frac{1}{4}|z|^2\right)^2 + t^2\right)^{s}\left|f(z,t)\right|^2dz\;dt.
\end{align*}
    \end{enumerate}
    \item homogeneous case: Let \(0 < s < 1\), then the following hold
    \begin{align*}
     \int\limits_{\Omega \times \mathbb{C}^n}\left( (2k+n)|\lambda|\right)^{-s}&   \left|\tilde{f}(a,w) \right|^2 d\nu_2(a) \;dw\\
     &\leq \lVert V_s \rVert  \left( \frac{2^{2n+3s}\Gamma(\frac{n+s}{2})^2}{\Gamma(1-s)\Gamma(\frac{n}{2})^2}\right)^{-1} \int\limits_{\mathbb{H}^n}\left|(z,t) \right|^{2s}\left(f(z,t)\right)^2dz\;dt. 
    \end{align*}
\end{itemize}
\end{theorem}

\section*{Concluding Remarks}
In this paper, we investigated weighted Pitt-type inequalities of the form
\begin{equation*}
        \left( \int\limits_{\Omega \times \mathbb{C}^n}|\tilde{f}(a,w)|^q u(a,w)^q dw\;d\nu_2(a)  \right )^{1/q}\leq C \left( \int\limits_{\mathbb{H}^n}|{f}(z,t)|^p v(z,t)^p dz\;dt  \right )^{1/p},
    \end{equation*}
    where \(\tilde{f}\) refers to the Strichartz Fourier transform on the Heisenberg group, \(u\) and \(v\) denote two non-negative measurable functions on \(\Omega \times \mathbb{C}^n\) and \(\mathbb{H}^n\) respectively.
    \begin{itemize}
        \item First by employing the Calder\'on's operator and the weighted Hardy's inequality we established the sufficient conditions for Pitt's inequality for the Strichartz Fourier transform on the Heisenberg group which is proved in Theorem \ref{theorem-Pitts for Hn}.
\item In section \ref{sec 4}, we used the estimate of first zero of the Laguerre polynomial \(L^{n-1}_k\) to derive the necessary conditions (given by \eqref{necessary conditions k>=1} and \eqref{necessary conditions k=0}) on weights which are implied by the Pitt's inequality. 
\item In section \ref{sec 5}, we obtain sufficient and necessary conditions on some non trivial weights.
    \end{itemize}
 Furthermore, we studied Paley's inequality for the Strichartz Fourier transform in section \ref{sec 6} and also examined the relationship between Hardy’s and Pitt’s inequalities in section \ref{sec 7}. After recalling their equivalence in the case of \(\mathbb{R}^n\) we obtained corresponding results for the case of \(\mathbb{H}^n\). More precisely, we obtained 
 \begin{itemize}
     \item Pitt's inequality \eqref{pitts with weight |z|-2s} for the conformal sublaplacian \(\mathcal{L}_s\) with weight \(|z|^{2s}\) which lead to Hardy's inequality \eqref{equiv hardy for weight |z|2s}. Rewriting in terms of the Strichartz Fourier transform we obtained Theorem \ref{theorem pitts weight |z|2s}.
     \item Pitt's inequality \eqref{pitts for weight |(z,t)|2s} for the conformal sublaplacian with weight \(|(z,t)|^{2s}\) which lead to Hardy's inequality \eqref{equiv hardy for weight |(z,t)|2s}. Rewriting in terms of the Strichartz Fourier transform we obtained Theorem \ref{theorem pitts weight |(z,t)|2s}.
     \item Another form of Hardy's inequality \eqref{hardy via trace hardy} and corresponding Pitt's inequality \eqref{pitts via trace hardy} is obtained adapting the techniques used in \cite{Roncal-Thangavelu-Trace_hardy_IMRN, corr-Roncal-Thangavelu-trace_hardy_imrn}. Moreover, bounds for sharp constant have been obtained.
     \item Using Hardy's inequality for the conformal sublaplacian and the fractional sublaplacian \(\mathcal{L}^s\) with non-homogeneous weight, we obtained Pitt type inequalities \eqref{pitts non-homogeneous 1} and \eqref{Pitts non-homogeneous 2}.
     \item using Hardy's inequality for the fractional sublaplacian with homogeneous weight, we obtained Pitt's inequality \eqref{Pitt's homogeneous}.
 \end{itemize}
 We would like to mention that the constants appearing in the deduced Pitt type inequalities are, however, not sharp.
 \section*{Acknowledgement}
 We would like to express our gratitude to Dr. Tapendu Rana for recommending an article that helped enhance a proof. The second author sincerely appreciates the Indian Institute of Technology Delhi for providing the Institute fellowship. Additionally, part of this work was carried out while the fourth author was visiting the Indian Institute of Technology Bombay, and he extends his gratitude for the hospitality provided by the institute.

\section*{Conflict of Interest}
The authors declare that they have no competing interests.

\bibliographystyle{acm}
\bibliography{ref_Pitts}

\begin{thebibliography}{10}

\bibitem{handbook}
{\sc Abramowitz, M., and Stegun, I.~A.}
\newblock {\em Handbook of mathematical functions with formulas, graphs, and mathematical tables}, vol.~No. 55 of {\em National Bureau of Standards Applied Mathematics Series}.
\newblock U. S. Government Printing Office, Washington, DC, 1964.

\bibitem{Beckner-Pitts_and_uncertainty_principle_1995}
{\sc Beckner, W.}
\newblock Pitt's inequality and the uncertainty principle.
\newblock {\em Proc. Amer. Math. Soc. 123}, 6 (1995), 1897--1905.

\bibitem{beckner-Pitts_with_sharp_convolution_estimates}
{\sc Beckner, W.}
\newblock Pitt's inequality with sharp convolution estimates.
\newblock {\em Proc. Amer. Math. Soc. 136}, 5 (2008), 1871--1885.

\bibitem{benedetto-heinig-weighted-inequalities}
{\sc Benedetto, J.~J., and Heinig, H.~P.}
\newblock Weighted {F}ourier inequalities: new proofs and generalizations.
\newblock {\em J. Fourier Anal. Appl. 9}, 1 (2003), 1--37.

\bibitem{bennett's_book}
{\sc Bennett, C., and Sharpley, R.}
\newblock {\em Interpolation of operators}, vol.~129 of {\em Pure and Applied Mathematics}.
\newblock Academic Press, Inc., Boston, MA, 1988.

\bibitem{bradley-weighted-hardy}
{\sc Bradley, J.~S.}
\newblock Hardy inequalities with mixed norms.
\newblock {\em Canad. Math. Bull. 21}, 4 (1978), 405--408.

\bibitem{Branson-inequalities_on_CR_Sphere}
{\sc Branson, T.~P., Fontana, L., and Morpurgo, C.}
\newblock Moser-{T}rudinger and {B}eckner-{O}nofri's inequalities on the {CR} sphere.
\newblock {\em Ann. of Math. (2) 177}, 1 (2013), 1--52.

\bibitem{calderon66}
{\sc Calder\'on, A.-P.}
\newblock Spaces between {$L\sp{1}$} and {$L\sp{\infty }$} and the theorem of {M}arcinkiewicz.
\newblock {\em Studia Math. 26\/} (1966), 273--299.

\bibitem{Pitts_for_quaternion_fourier_transform}
{\sc Chen, L.-P., Kou, K.~I., and Liu, M.-S.}
\newblock Pitt's inequality and the uncertainty principle associated with the quaternion {F}ourier transform.
\newblock {\em J. Math. Anal. Appl. 423}, 1 (2015), 681--700.

\bibitem{thangavelu-roncal-oscar-Hardy_inequalities}
{\sc Ciaurri, O., Roncal, L., and Thangavelu, S.}
\newblock Hardy-type inequalities for fractional powers of the {D}unkl-{H}ermite operator.
\newblock {\em Proc. Edinb. Math. Soc. (2) 61}, 2 (2018), 513--544.

\bibitem{Cowling-comp_bdd_multipliers}
{\sc Cowling, M., and Haagerup, U.}
\newblock Completely bounded multipliers of the {F}ourier algebra of a simple {L}ie group of real rank one.
\newblock {\em Invent. Math. 96}, 3 (1989), 507--549.

\bibitem{Cowling-Price-HPW-inequality}
{\sc Cowling, M.~G., and Price, J.~F.}
\newblock Bandwidth versus time concentration: the {H}eisenberg-{P}auli-{W}eyl inequality.
\newblock {\em SIAM J. Math. Anal. 15}, 1 (1984), 151--165.

\bibitem{Pitts_and_Boas_for_Fourier_and_Hankel}
{\sc De~Carli, L., Gorbachev, D., and Tikhonov, S.}
\newblock Pitt and {B}oas inequalities for {F}ourier and {H}ankel transforms.
\newblock {\em J. Math. Anal. Appl. 408}, 2 (2013), 762--774.

\bibitem{Pitts_inequalities_Carli-Tikhonov}
{\sc De~Carli, L., Gorbachev, D., and Tikhonov, S.}
\newblock Pitt inequalities and restriction theorems for the {F}ourier transform.
\newblock {\em Rev. Mat. Iberoam. 33}, 3 (2017), 789--808.

\bibitem{Frank_and_Lieb-sharp_constants_on_Hn}
{\sc Frank, R.~L., and Lieb, E.~H.}
\newblock Sharp constants in several inequalities on the {H}eisenberg group.
\newblock {\em Ann. of Math. (2) 176}, 1 (2012), 349--381.

\bibitem{Tikhonov-weighted_norm_integral_transforms}
{\sc Gorbachev, D., Liflyand, E., and Tikhonov, S.}
\newblock Weighted norm inequalities for integral transforms.
\newblock {\em Indiana Univ. Math. J. 67}, 5 (2018), 1949--2003.

\bibitem{grafakos_book}
{\sc Grafakos, L.}
\newblock {\em Classical {F}ourier analysis}, second~ed., vol.~249 of {\em Graduate Texts in Mathematics}.
\newblock Springer, New York, 2008.

\bibitem{heinig_1984}
{\sc Heinig, H.~P.}
\newblock Weighted norm inequalities for classes of operators.
\newblock {\em Indiana Univ. Math. J. 33}, 4 (1984), 573--582.

\bibitem{Hormander-1960-paley}
{\sc H\"ormander, L.}
\newblock Estimates for translation invariant operators in {$L\sp{p}$}\ spaces.
\newblock {\em Acta Math. 104\/} (1960), 93--140.

\bibitem{weighted-fourier-symmetric-spaces-Pusti}
{\sc Kumar, P., Pusti, S., Rana, T., and Singh, M.}
\newblock Weighted fourier inequalities and application of restriction theorems on rank one riemannian symmetric spaces of noncompact type.
\newblock {\em arXiv:2406.06268\/} (2024).

\bibitem{Lieb-Loss}
{\sc Lieb, E.~H., and Loss, M.}
\newblock {\em Analysis}, second~ed., vol.~14 of {\em Graduate Studies in Mathematics}.
\newblock American Mathematical Society, Providence, RI, 2001.

\bibitem{Liflyand-Tikhonov}
{\sc Liflyand, E., and Tikhonov, S.}
\newblock Two-sided weighted {F}ourier inequalities.
\newblock {\em Ann. Sc. Norm. Super. Pisa Cl. Sci. (5) 11}, 2 (2012), 341--362.

\bibitem{frank-oliver-handbook}
{\sc Olver, F.~W.}
\newblock {\em NIST handbook of mathematical functions hardback and CD-ROM}.
\newblock Cambridge university press, 2010.

\bibitem{Pitts-1937}
{\sc Pitt, H.~R.}
\newblock Theorems on {F}ourier series and power series.
\newblock {\em Duke Math. J. 3}, 4 (1937), 747--755.

\bibitem{Roncal-Chunxia-Sharp_Lp_Hardy_Stein_Weiss}
{\sc Roncal, L., and Tao, C.}
\newblock Sharp \({L}^p\) {H}ardy and {S}tein-{W}eiss inequalities on the {H}eisenberg group.
\newblock (preprint).

\bibitem{Roncal-thangavelu-Hardy_inequality}
{\sc Roncal, L., and Thangavelu, S.}
\newblock Hardy's inequality for fractional powers of the sublaplacian on the {H}eisenberg group.
\newblock {\em Adv. Math. 302\/} (2016), 106--158.

\bibitem{Roncal-Thangavelu-Trace_hardy_IMRN}
{\sc Roncal, L., and Thangavelu, S.}
\newblock An extension problem and trace {H}ardy inequality for the sub-{L}aplacian on {$H$}-type groups.
\newblock {\em Int. Math. Res. Not. IMRN}, 14 (2020), 4238--4294.

\bibitem{corr-Roncal-Thangavelu-trace_hardy_imrn}
{\sc Roncal, L., and Thangavelu, S.}
\newblock Corrigendum to: {A}n extension problem and trace {H}ardy inequality for the sublaplacian on {$H$}-type groups.
\newblock {\em Int. Math. Res. Not. IMRN}, 12 (2022), 9598--9602.

\bibitem{book-fractional_integrals}
{\sc Samko, S.~G., Kilbas, A.~A., and Marichev, O.~I.}
\newblock {\em Fractional integrals and derivatives}.
\newblock Gordon and Breach Science Publishers, Yverdon, 1993.
\newblock Theory and applications,.

\bibitem{szego}
{\sc Szeg\"o, G.}
\newblock {\em Orthogonal {P}olynomials}, vol.~Vol. 23 of {\em American Mathematical Society Colloquium Publications}.
\newblock American Mathematical Society, New York, 1939.

\bibitem{thangavelu-uncertainty-book}
{\sc Thangavelu, S.}
\newblock {\em An introduction to the uncertainty principle}, vol.~217 of {\em Progress in Mathematics}.
\newblock Birkh\"auser Boston, Inc., Boston, MA, 2004.
\newblock Hardy's theorem on Lie groups, With a foreword by Gerald B.\ Folland.

\bibitem{thangaveluheisenberg}
{\sc Thangavelu, S.}
\newblock {\em Harmonic analysis on the Heisenberg group}, vol.~159.
\newblock Springer Science \& Business Media, 2012.

\bibitem{Strichartz_Fourier_thangavelu}
{\sc Thangavelu, S.}
\newblock A scalar-valued {F}ourier transform for the {H}eisenberg group.
\newblock In {\em From classical analysis to analysis on fractals. {V}ol. 1. {A} tribute to {R}obert {S}trichartz}, Appl. Numer. Harmon. Anal. Birkh\"auser/Springer, Cham, (2023), pp.~137--163.

\bibitem{wong_weyl-transforms}
{\sc Wong, M.~W.}
\newblock {\em Weyl transforms}.
\newblock Universitext. Springer-Verlag, New York, 1998.

\end{thebibliography}
\end{document}